\documentclass[12pt]{amsart}
\advance\hoffset by -0.5 truecm
\usepackage{amsmath,amscd}
\usepackage{amssymb}
\usepackage{amsthm}
\newtheorem{Theorem}{Theorem}[section]
\newtheorem*{TheoremNoNumber}{Theorem}
\newtheorem{Lemma}[Theorem]{Lemma}
\newtheorem{Corollary}[Theorem]{Corollary}

\newtheorem{Proposition}[Theorem]{Proposition}

\newtheorem{Definition}[Theorem]{Definition}
\newtheorem*{DefinitionNoNumber}{Definition}

\newtheorem{Remark}[Theorem]{Remark}

\def \dim{{\mbox {dim}}\,}

\def\V{\mbox{Var}}

\def\Z{{\mathbb Z}}
\def\R\re
\def\V{\bf V}

\def \re{{\mathbb R}}

\def \C{{\mathbb C}}

\def \V{{\bf V}}

\def \H{\mathcal H}

\newcommand{\abs}[1]{\lvert #1 \rvert}
\newcommand{\norm}[1]{\lVert #1 \rVert}

\newcommand{\mDp}{\mathcal{D}'}

\newcommand{\dbar}{\overline{\partial}}

\begin{document}
\title[Spectral rigidity and invariant distributions on Anosov surfaces]{Spectral rigidity and invariant distributions on Anosov surfaces}

\author[G.P. Paternain]{Gabriel P. Paternain}
\address{ Department of Pure Mathematics and Mathematical Statistics,
University of Cambridge,
Cambridge CB3 0WB, UK}
\email {g.p.paternain@dpmms.cam.ac.uk}

\author[M. Salo]{Mikko Salo}
\address{Department of Mathematics and Statistics, University of Jyv\"askyl\"a}
\email{mikko.j.salo@jyu.fi}

\author[G. Uhlmann]{Gunther Uhlmann}
\address{Department of Mathematics, University of Washington}

\email{gunther@math.washington.edu}




\begin{abstract}
This article considers inverse problems on closed Riemannian surfaces whose geodesic flow is Anosov. We prove spectral rigidity for any Anosov surface and injectivity of the geodesic ray transform on solenoidal 2-tensors. We also establish surjectivity results for the adjoint of the geodesic ray transform on solenoidal tensors. The surjectivity results are of independent interest and imply the existence of many geometric invariant distributions on the unit sphere bundle. In particular, we show that on any Anosov surface $(M,g)$, given a smooth function $f$ on $M$ there is a distribution in the Sobolev space $H^{-1}(SM)$ that is invariant under the geodesic flow and whose projection to $M$ is the given function $f$.

\end{abstract}

\maketitle

\section{Introduction} \label{sec_intro}

Let $(M,g)$ be a closed oriented Riemannian manifold with geodesic flow $\phi_t$ acting on the unit sphere bundle $SM$.
Recall that the geodesic flow is said to be {\it Anosov} if there
is a continuous invariant splitting
$TSM=E^0\oplus E^{u}\oplus E^{s}$, where $E^0$ is the flow direction, and 
there are constants $C>0$ and $0<\rho<1<\eta$ such that 
for all $t>0$ 
\[\|d\phi_{-t}|_{E^{u}}\|\leq C\,\eta^{-t}\;\;\;\;\mbox{\rm
and}\;\;\;\|d\phi_{t}|_{E^{s}}\|\leq C\,\rho^{t}.\]
We will say that $(M,g)$ is Anosov, if its geodesic flow is Anosov.
It is very well known that the geodesic flow of a closed
negatively curved Riemannian manifold is a contact Anosov
flow \cite{KH}. The Anosov property automatically implies that the manifold is
free of conjugate points \cite{K,A,Man} and absence of conjugate points simply means that
between two points in the universal covering of $M$ there is a unique geodesic
connecting them.

There is a purely Riemannian way of characterizing this uniform hyperbolicity
of the geodesic flow which is relevant for us \cite{Ru}:   $(M,g)$ is Anosov if and only if
the metric $g$ lies in the $C^2$-interior of the set of metrics without conjugate points.
One reason for mentioning this characterization is to motivate the present results in terms of an interesting analogy between Anosov manifolds (that have no boundary) and compact {\it simple} manifolds with boundary.
Recall that a compact oriented Riemannian manifold $(M,g)$ is said to be simple if its boundary
is strictly convex and any two points are joined by a unique geodesic depending smoothly on the end points.
The notion of simple manifold appears naturally in the context of the boundary rigidity problem \cite{Mi}
and it has been at the center of recent activity on geometric inverse problems. As in the Anosov case, simple
manifolds are free of conjugate points (this follows directly from the definition) and are $C^2$-stable under perturbations.

\subsection{Ray transforms and spectral rigidity}

Inverse problems frequently lead to the study of geodesic ray transforms. These transforms could be acting on functions, or more generally
on tensors depending on the problem at hand. We consider here
the geodesic ray transform acting on symmetric tensor fields on $M$. Given a symmetric (covariant) $m$-tensor field $f = f_{i_1 \cdots i_m} \,dx^{i_1} \otimes \cdots \otimes \,dx^{i_m}$ on $M$, we define the corresponding function on $SM$ by 
$$
f(x,v) = f_{i_1 \cdots i_m} v^{i_1} \cdots v^{i_m}.
$$

Let us consider first the case of simple manifolds with boundary. Geodesics going from $\partial M$ into $M$ are parametrized by $\partial_+ (SM) = \{(x,v) \in SM \,;\, x \in \partial M, \langle v,\nu \rangle \leq 0 \}$ where $\nu$ is the outer unit normal vector to $\partial M$. For $(x,v) \in SM$ we let $t \mapsto \gamma(t,x,v)$ be the geodesic starting from $x$ in direction $v$. The ray transform of $f$ is defined by 
$$
I_mf(x,v) = \int_0^{\tau(x,v)} f(\phi_t(x,v)) \,dt, \quad (x,v) \in \partial_+(SM),
$$
where $\tau(x,v)$ is the exit time of $\gamma(t,x,v)$.
If $h$ is a symmetric $(m-1)$-tensor field, its inner derivative $dh$ is a symmetric $m$-tensor field defined by $dh=\sigma\nabla h$, where $\sigma$ denotes
symmetrization and $\nabla$ is the Levi-Civita connection. It is easy to see that
$$
dh(x,v)= Xh(x,v),
$$
where $X$ is the geodesic vector field associated with $\phi_t$.
If additionally $h|_{\partial M} = 0$, then clearly $I_m(dh) = 0$. The transform $I_m$ is said to be \emph{$s$-injective} if these are the only elements in the kernel. The terminology arises from the fact that
any tensor field $f$ may be written uniquely as $f=f^s+dh$, where $f^s$
is a symmetric $m$-tensor with zero divergence and $h$ is an $(m-1)$-tensor
with $h|_{\partial M} = 0$ (cf. \cite{Sh}). The tensor fields $f^s$ and
$dh$ are called respectively the {\it solenoidal} and {\it potential} parts
of $f$. Saying that $I_m$ is $s$-injective is saying precisely that
$I_m$ is injective on the set of solenoidal tensors.

In \cite{PSU_tensor} we proved that when $(M,g)$ is a simple surface, then $I_m$ is $s$-injective. Here we would like to investigate
the analogous tensor tomography problem when $(M,g)$ is a closed Anosov surface.  The analogy proceeds as follows.
Let $\mathcal{G}$ be the set of closed geodesics on $(M,g)$, parametrized by arc length. The ray transform of a symmetric $m$-tensor field $f$ on $M$ is defined by 
$$
I_{m}f(\gamma) = \int_0^T f(\gamma(t),\dot{\gamma}(t)) \,dt, \quad \gamma \in \mathcal{G} \text{ has period } T.
$$
As before $I_{m}(dh)(\gamma) = 0$ for all $\gamma \in \mathcal{G}$ if $h$ is a symmetric $(m-1)$-tensor. The question of $s$-injectivity is whether these are the only tensors
in the kernel of $I_m$.

Our first main result is:

\begin{Theorem} Let $(M,g)$ be a closed oriented Anosov surface. Then $I_{2}$ is $s$-injective.
\label{thm:I2}
\end{Theorem}

A basic inverse problem in spectral geometry, inspired by the famous question ``Can you hear the shape of a drum?'' of M.~Kac \cite{Kac}, is to determine properties of a compact Riemannian manifold $(M,g)$ from the spectrum $\text{Spec}(-\Delta_g)$ of the Laplace-Beltrami operator (with Dirichlet boundary condition if the manifold has nonempty boundary). Two Riemannian manifolds are said to be \emph{isospectral} if their spectra and also the multiplicities of eigenvalues coincide. There is a large literature on isospectral manifolds with both positive results and counterexamples: we refer to the survey \cite{DaHe} for positive results and \cite{Go, GoPeSc} for negative ones.

In particular, for manifolds with no boundary, there are examples of isospectral but non-isometric manifolds even having constant negative sectional curvature \cite{Sun, Vig}. On the other hand, one has \emph{local audibility} for metrics of constant negative sectional curvature \cite{Sh_audibility}, meaning that any such metric $g$ has a $C^{\infty}$ neighborhood where $g$ is uniquely spectrally determined. For metrics of variable negative curvature, local audibility is an open question even in two dimensions. However, \emph{spectral rigidity} is known: any isospectral smooth family $(g_s)$ where $s \in (-\varepsilon,\varepsilon)$ and $g_0$ has negative curvature must satisfy $g_s = g_0$ up to isometry \cite{CS, GK}. There are also compactness results stating that the set of metrics isospectral to a negative curvature metric $g$ is precompact in the $C^{\infty}$ topology up to isometry \cite{BrPePe, OsPhSa}.

By the work of Guillemin and Kazhdan \cite{GK}, we obtain the following spectral rigidity result as a consequence of Theorem \ref{thm:I2}.

\begin{Theorem} \label{thm_spectralrigidity_main}
Let $(M,g)$ be a closed oriented Anosov surface. If $(g_s)$ is a smooth family of Riemannian metrics on $M$ for $s \in (-\varepsilon,\varepsilon)$ such that $g_0 = g$ and the spectra of $-\Delta_{g_s}$ coincide up to multiplicity,
$$
\text{Spec}(-\Delta_{g_s}) = \text{Spec}(-\Delta_{g_0}), \quad s \in (-\varepsilon,\varepsilon),
$$
then there exists a family of diffeomorphisms $\psi_s: M \to M$ with $\psi_0 = \text{Id}$ and 
$$
g_s = \psi_s^* g_0.
$$
\end{Theorem}

The work of Guillemin and Kazhdan implicitly uses that along an isospectral family $g_s$ (with $g_0$ Anosov),  the marked length spectrum remains unchanged. Recall that the marked length spectrum is the function which associates to each free homotopy class, the length of the unique closed geodesic representing the class. J.-P. Otal \cite{Otal} and C. Croke \cite{Croke} have independently shown (with different methods) that two negatively curved surfaces with the same marked length spectrum must be isometric. It is reasonable to expect that a similar result should hold for the larger open set of Anosov metrics, but both proofs seem to use in a rather substantial way the sign of the curvature and at the time of writing we do not know how to address this non-linear problem. Also, we mention that the work  of Otal involves geodesic currents that are somewhat related to the invariant distributions used in this paper.

To state the results on $s$-injectivity of $I_{m}$ for $m\geq 3$, we first give a definition involving conjugate points for a modified Jacobi equation. Here $K$ is the Gaussian curvature.

\begin{DefinitionNoNumber} Let $(M,g)$ be a closed oriented Riemannian surface. We say that $(M,g)$ is free of $\beta$-conjugate points if for any geodesic $\gamma(t)$, all nontrivial solutions of the equation $\ddot{y} + \beta K(\gamma(t)) y = 0$ with $y(0) = 0$ only vanish at $t = 0$. The \emph{terminator value} of $(M,g)$ is defined to be 
$$
\beta_{Ter} = \sup \,\{ \beta \in [0,\infty]:\, \text{$(M,g)$ is free of $\beta$-conjugate points} \}.
$$
\end{DefinitionNoNumber}

Clearly $1$-conjugate points correspond to conjugate points in the usual sense. For a closed oriented surface $(M,g)$, we will show in Section \ref{sec_betaconjugate} that 
\begin{itemize}
\item if $(M,g)$ is free of $\beta_0$-conjugate points for some $\beta_0 > 0$, then $(M,g)$ is free of $\beta$-conjugate points for $\beta \in [0,\beta_0]$, 
\item 
$(M,g)$ is Anosov if and only if $\beta_{Ter} > 1$ and there is no geodesic trapped in the region of zero Gaussian curvature (see Corollary \ref{corollary:ganosov} below; this seems to be a new geometric characterization of the Anosov property generalizing \cite[Corollary 3.6]{Ebe});
\item if $(M,g)$ has no focal points (see definition below), then $\beta_{Ter} \geq 2$;
\item 
$(M,g)$ has nonpositive curvature if and only if $\beta_{Ter} = \infty$.
\end{itemize}

\begin{Theorem} \label{thm_sinjectivity_main}
Let $(M,g)$ be a closed oriented surface such that no geodesic is trapped
in the region of zero Gaussian curvature. Suppose in addition that $\beta_{Ter}\geq (m+1)/2$, where
$m$ is an integer $\geq 2$. Then $I_m$ is $s$-injective.
\end{Theorem}

This theorem was proved earlier for $m=0,1$ \cite{DS} ($\dim M$ arbitrary), for the case $m=2$ if additionally the surface has no focal points \cite{SU}, and for $m \geq 2$ if the surface has negative curvature \cite{GK}. In \cite{CS} the theorem was proved for non-positive curvature and $\dim M$ arbitrary and it is also known that the kernel of $I_m$ is finite dimensional \cite{DS}. 
Certainly, for $m=2$, Theorem \ref{thm_sinjectivity_main} is weaker than Theorem \ref{thm:I2}, nevertheless even this weaker version is new.
In Section \ref{sec_examples} we provide open sets of Anosov surfaces with $3/2 \leq \beta_{Ter} < 2$, thus showing that Theorem \ref{thm_sinjectivity_main} for $m=2$ already improves the main result of \cite{SU}.
One could take the  view point that the more refined argument which is involved in the proof of
Theorem \ref{thm:I2} deals with the harder case when $\beta_{Ter}\in (1,3/2)$.
However, at the moment on Anosov surfaces we need the additional condition $\beta_{Ter} \geq (m+1)/2$ for $m\geq 3$. This condition is closely related to the works \cite{Pestov, Dairbekov_nonconvex} where absence of $\beta$-conjugate points also appears in the case of manifolds with boundary.

\subsection{Invariant distributions and surjectivity of $I_m^*$}

When proving $s$-injectivity of the geodesic ray transform on both simple and Anosov manifolds, a first step is to consider the {\it transport equation} (or {\it cohomological equation}). If $I_{m}(f)=0$ it is possible to show the existence
of a smooth function $u:SM\to\re$ such that
\[Xu=f\]
and $u|_{\partial(SM)}=0$ (for closed manifolds this condition is empty).  In the Anosov case, this is a consequence of one of the celebrated
Livsic theorems \cite{L1,L2} together with the regularity addendum from \cite{dLMM}. For surfaces of negative curvature the existence
of a smooth solution to the transport equation was first proved by Guillemin and Kazhdan in \cite{GK1}, motivated by spectral rigidity for such surfaces \cite{GK}. 

The main result in \cite{PSU_tensor} admits the following extension which exposes the various ingredients needed to solve the tensor
tomography problem for a simple surface. Recall that a surface is said to be non-trapping if every geodesic reaches the boundary in finite time
(perhaps the correct replacement of this notion in the case of closed manifolds is ergodicity of the geodesic flow).
Let $C^{\infty}_{\alpha}(\partial_+(SM))$ denote the set of functions $h\in C^{\infty}(\partial_{+}(SM))$ such that the
unique solution $w$ to $Xw=0$, $w|_{\partial_{+}(SM)}=h$ is smooth.
In natural $L^2$ inner products, the adjoint of $I_0$ is the operator 
$$I_0^*: C^{\infty}_{\alpha}(\partial_+(SM)) \to C^{\infty}(M), \ \ I_0^* h(x) = \int_{S_x} w(x,v) \,dS_x(v).$$
Here $S_x = \{ (x,v) \in TM \,;\, \abs{v} = 1 \}$ and $dS_x$ is the volume form on $S_x$. For more details see \cite{PU}, where it is also proved that the adjoint $I_0^*$ is surjective on any simple manifold.

\begin{TheoremNoNumber}[\cite{PSU_tensor}] Let $(M,g)$ be a compact nontrapping surface with strictly convex smooth boundary. Suppose in addition that $I_0$ and $I_1$ are $s$-injective and that $I_0^*$ is surjective. Then $I_m$ is $s$-injective for $m \geq 2$. 
\end{TheoremNoNumber}

We already mentioned that $I_0$ and $I_{1}$ are known to be $s$-injective for an Anosov surface and one of the purposes of the present
paper is to show that  $I_{0}^*$ is surjective. To discuss the adjoint it is convenient to give a brief preliminary discussion.

For the following facts on function spaces we refer to \cite{DeRham, Schwartz}. Denoting by $\delta_{\gamma}$ the measure on $SM$ which corresponds to integrating over the curve $(\gamma(t),\dot{\gamma}(t))$ on $SM$, we have in the distributional pairing 
$$
If(\gamma) = \langle \delta_{\gamma}, f \rangle, \quad \gamma \in \mathcal{G}.
$$
Denote by $\mDp(SM)$ the set of distributions (continuous linear functionals) on $C^{\infty}(SM)$, and equip this space with the weak$^*$ topology. These spaces are reflexive, so the dual of $\mDp(SM)$ is $C^{\infty}(SM)$. The geodesic vector field $X$ acts on $\mDp(SM)$ by duality (since it is a differential operator with smooth coefficients). We consider the set of invariant distributions (a closed subspace of $\mDp(SM)$), 
$$
\mDp_{\text{inv}}(SM) = \{ \mu \in \mDp(SM) \,;\, X\mu = 0 \}.
$$
Thus $\mu \in \mDp(SM)$ is invariant iff $\langle \mu, X\varphi \rangle = 0$ for all $\varphi \in C^{\infty}(SM)$. Now the set $\{ \delta_{\gamma} \,;\, \gamma \in \mathcal{G} \}$ is dense in $\mDp_{\text{inv}}(SM)$, since if $f \in C^{\infty}(SM)$ satisfies $\langle \delta_{\gamma}, f \rangle = 0$ for all $\gamma \in \mathcal{G}$, then by the Livsic theorem $f = Xu$ for some $u \in C^{\infty}(SM)$ and consequently $\langle \mu, f \rangle = 0$ for all $\mu \in \mDp_{\text{inv}}(SM)$.

It follows that we may without loss of generality define $I$ as the map 
$$
I: C^{\infty}(SM) \to L(\mDp_{\text{inv}}(SM),\re), \ \ If(\nu) = \langle \nu, f \rangle \quad \text{for } \nu \in \mDp_{\text{inv}}(SM).
$$
Here $L(E,\re)$ denotes the set of continuous linear maps from a locally convex topological vector space $E$ to $\re$. Equipping this set with the weak$^*$ topology, it follows that $I$ is a continuous linear map from the Frech\'et space $C^{\infty}(SM)$ into the locally convex space $L(\mDp_{\text{inv}}(SM),\re)$.
Since $\mDp_{\text{inv}}(SM)$ is reflexive as a closed subspace of a reflexive space, the dual of $L(\mDp_{\text{inv}}(SM),\re)$ is $\mDp_{\text{inv}}(SM)$. Therefore the adjoint of $I$ is the map 
$$
I^*: \mDp_{\text{inv}}(SM) \to \mDp(SM), \ \ \langle I^* \nu, \varphi \rangle = \langle \nu, I\varphi \rangle \quad \text{for } \varphi \in C^{\infty}(SM).
$$
Restricting the domain of $I$ gives rise for instance to the ray transform on $0$-forms, 
$$
I_0: C^{\infty}(M) \to L(\mDp_{\text{inv}}(SM),\re), \ \ I_0 f(\nu) = \langle \nu, f \circ \pi \rangle
$$
where $\pi: SM \to M$ is the natural projection. The adjoint of this map is 
$$
I_0^*: \mDp_{\text{inv}}(SM) \to \mDp(M), \ \ I_0^* \nu = 2\pi \nu_0
$$
where for any $\mu \in \mDp(SM)$, the average $\mu_0$ is the element in $\mDp(M)$ given by $\langle \mu_0, \psi \rangle = \frac{1}{2\pi} \langle \mu, \psi \circ \pi \rangle$ for $\psi \in C^{\infty}(M)$. 
On an oriented surface (see Section \ref{sec_preliminaries})
any smooth function $u\in C^{\infty}(SM)$ admits a Fourier expansion $u=\sum_{m\in\Z}u_m$ where
\[u_{m}(x,v):=\frac{1}{2\pi}\int_{0}^{2\pi}u(\rho_{t}(x,v))e^{-imt}\,dt\]
and $\rho_{t}$ is the flow of the vertical vector field $V$ determined by the principal circle fibration $\pi:SM\to M$. 
Similarly, distributions admit Fourier expansions as above, and $\mu_0$ is just the zeroth Fourier coefficient.
We can now state our next result, which expresses the surjectivity of $I_0^*$ in terms of the existence of invariant distributions.

\begin{Theorem} \label{thm_main_i0star}
Let $(M,g)$ be an Anosov surface. Given $f \in C^{\infty}(M)$, there exists $w \in H^{-1}(SM)$ with $Xw = 0$ and $w_0 = f$.
Moreover if we write $w=\sum_{k\in\Z}w_k$, then $w_k\in C^{\infty}(SM)$ for all even $k$.
\end{Theorem}

Note that there are no $L^2$ solutions to $Xw=0$ (not even $L^1$) due to the ergodicity of the geodesic flow \cite{A0,H0}, so $H^{-1}$ is the optimal regularity in the $H^k$ Sobolev scale. This is a crucial difference with the boundary case. Using Theorem \ref{thm_main_i0star} one can show
as in \cite{PSU_tensor} that given any 1-form $A$ on $M$ orthogonal to the space of harmonic 1-forms, there is $w\in H^{-1}(SM)$ which is holomorphic in the velocity variable
(i.e. $w_k=0$ for all $k<0$) for which $Xw=A$. These holomorphic integrating factors are the key to proving $s$-injectivity on simple surfaces in \cite{PSU_tensor}, but unfortunately
we have been unable to put to use their distributional version in the Anosov case.

In \cite[Theorem 4.2]{PU2}  the authors show the surjectivity of $I^*_{1}$ for compact simple manifolds. The version for Anosov surfaces
is as follows.  We say that a $1$-form $A$ is solenoidal if it has zero divergence. 

\begin{Theorem} \label{thm_main_i1star}
Let $(M,g)$ be an Anosov surface and let $A$ be a solenoidal 1-form. Then there exists $w\in H^{-1}(SM)$ such that
$Xw=0$ and $w_{-1}+w_{1}=A$.
Moreover if we write $w=\sum_{k\in\Z}w_k$, then $w_k\in C^{\infty}(SM)$ for all odd $k$.
\end{Theorem}

As we will see in Section \ref{proofmain}, this result will be the key for proving solenoidal injectivity of $I_{2}$.

Next let us discuss surjectivity of $I^*_{m}$ for $m\geq 2$. The conformal class of the Riemannian metric $g$ determines
a complex structure.
Given a positive integer $m$, let $\mathcal H_m$ denote the space of
holomorphic sections of the
$m$-th power of the canonical line bundle. By the Riemann-Roch theorem 
this space has complex dimension $(2m-1)(\mbox{\rm \tt{g}}-1)$ for $m\geq 2$
and complex dimension $\tt{g}$ for $m=1$, where $\tt{g}$ is the genus of $M$.
(For $m=1$ we get the holomorphic 1-forms and for $m=2$ the holomorphic
quadratic differentials.) Note that the elements in $\mathcal H_m$ can be regarded
as functions on $SM$.\footnote{Sections of the $m$-th power of the canonical
line bundle can be regarded as functions on $SM$ which transform according to the rule $f(x,\rho_{t}(x,v))=e^{imt}f(x,v)$.}

\begin{Theorem} \label{thm_main_imstar}
Let $(M,g)$ be a closed oriented surface having no geodesic trapped in the set of zero curvature.
\begin{enumerate}
\item If $(M,g)$ has no focal points, or more generally if $\beta_{Ter} > 3/2$, then given $q\in \H_{2}$
there exists $w\in H^{-1}(SM)$ such that $Xw=0$ and $w_{2}=q$.
Moreover, $w_{2j}\in C^{\infty}(SM)$ for all $j\geq 1$.
\item If $(M,g)$ has nonpositive curvature, or more generally if $\beta_{Ter}> (m+1)/2$ where $m \geq 2$, then given $q\in \H_{m}$ there exists $w\in H^{-1}(SM)$ such that $Xw=0$ and $w_{m}=q$.
Moreover if $m$ is even, $w_{2j}\in C^{\infty}(SM)$ for all $j\geq m/2$.
Similarly, if $m$ is odd,  $w_{2j+1}\in C^{\infty}(SM)$ for all $j\geq (m-1)/2$.
\end{enumerate}
\end{Theorem}

Theorem \ref{theorem:sur3} below gives a form of item (1) in Theorem \ref{thm_main_imstar}  which applies to any non hyperelliptic Anosov surface, but it only gives $H^{-5}$ regularity for the invariant distribution.

Recall that a Riemannian manifold is said to have no focal points if for every unit speed geodesic $\gamma(t)$ and every non-zero Jacobi field $J(t)$ along $\gamma$ with $J(0)=0$, the function $t\mapsto |J(t)|^2$ has positive derivative for $t>0$. Geometrically, this means that the manifold has no conjugate points and geodesic balls in the universal covering are strictly convex. It is easy to check that
a manifold with non-positive sectional curvature has no focal points.

A result of P. Eberlein \cite{Ebe}
asserts that a surface with no focal points is Anosov if and only if every geodesic hits a point of negative Gaussian curvature and using this
it is possible to produce Anosov surfaces of non-positive curvature which have open sets with zero Gaussian curvature \cite{Ebe}.
There are also examples of Anosov surfaces isometrically embedded in $\re^3$ \cite{DP} and Anosov surfaces with focal points \cite{G}.

The existence of distributions as in Theorems \ref{thm_main_i0star}--\ref{thm_main_imstar} was first established by Guillemin and Kazhdan in \cite{GK1} for surfaces
of negative curvature, but as far as we can see their proof does not extend to the Anosov case; moreover the precise regularity of the distributions
was not considered there.  In general, an arbitrary transitive Anosov flow has a plethora of invariant measures and distributions, but the ones in Theorems \ref{thm_main_i0star}--\ref{thm_main_imstar}  are {\it geometric} since they really depend on the geometry of the circle fibration $\pi:SM\to M$.
In the case of surfaces of constant negative curvature these distributions and their regularity are discussed in \cite[Section 2]{AZ}.

Finally our methods also give new results for the transport equation, for example:

\begin{Theorem} \label{thm_main_transport}
Let $(M,g)$ be a closed surface of genus $\geq 2$ without focal points. Let $f$ be a symmetric $m$-tensor  with $m\leq 3$ and assume
that there is a smooth solution $u$ to $Xu=f$. Then $f$ is a potential tensor.
\end{Theorem}

Note that in this theorem we do not need to assume that $(M,g)$ is Anosov, but if it is, then combining this result with the
Livsic theorem we obtain right away that $I_3$ is $s$-injective. 

This paper is organized as follows. Section \ref{sec_intro} is the introduction, and Section \ref{sec_preliminaries} contains some preliminaries on Fourier analysis on the unit sphere bundle and the basic energy identity, called Pestov identity, that will be used below. In Section \ref{sec_control} we introduce $\alpha$-controlled surfaces motivated by the Pestov identity. Sections \ref{surjectiveadjoint} and \ref{sec_surjective2} contain the proofs of the surjectivity results for $I_m^*$, based on subelliptic estimates for certain (non-local if $m \geq 1$) second order operators on $SM$, and Section \ref{sec_injectivity} gives the corresponding injectivity results. In Section \ref{sec_betaconjugate} we consider $\beta$-conjugate points and hyperbolicity of related cocycles, leading to a sufficient condition for the injectivity and surjectivity results, and Section \ref{sec_examples} is devoted to examples.  Section \ref{proofmain} contains the proof of Theorem \ref{thm:I2} which builds on the surjectivity result for $I_{1}^{*}$.
Finally we mention that there are versions of Theorems \ref{thm_main_i0star} and \ref{thm_main_i1star}  and of the results in Section \ref{sec_betaconjugate} in any dimensions, but we shall consider these elsewhere.

\subsection*{Acknowledgements}

M.S. was supported in part by the Academy of Finland and an ERC starting grant, and G.U. was partly supported by NSF and a Walker Family Endowed Professorship. The authors would like to express their gratitude to the Fields Institute and the organizers of the program on Geometry in Inverse Problems in 2012 where part of this work was carried out. We would also like to thank Pelham Wilson for very useful discussions concerning Max Noether's theorem which is used in the proof of Theorem \ref{thm:I2}.

\section{Preliminaries} \label{sec_preliminaries}
Let $(M,g)$ be a closed oriented surface with unit circle bundle $SM$.
Let $X$ be the geodesic vector field on $SM$, and let $V$ be the vertical vector field. We let $X_{\perp} = [X,V]$. 
There are two additional structure equations given by $X=[V,X_{\perp}]$ and $[X,X_{\perp}]=-KV$, where $K$ is the Gaussian curvature.

There is an orthogonal decomposition of $L^2(SM)$ given by 
$$
L^2(SM) = \bigoplus_{k=-\infty}^{\infty} H_k
$$
where $H_k$ is the eigenspace of $-iV$ corresponding to the eigenvalue $k$. Let also $\Omega_k = H_k \cap C^{\infty}(SM)$. If $f \in L^2(SM)$ we write $f = \sum_{k=-\infty}^{\infty} f_k$ where $f_k \in H_k$. Then $\norm{f}^2 = \sum \norm{f_k}^2$ where 
$$
(u,v) = \int_{SM} u\bar{v} \,d(SM), \quad \norm{u} = (u,u)^{1/2}.
$$
The volume form $d(SM)$ is uniquely determined by the requirement that it takes the value $1$
on the frame $\{X,X_{\perp},V\}$. The volume form is preserved by the three vector fields in the frame.
If $x = (x_1,x_2)$ are isothermal coordinates in $(M,g)$ and if $\theta$ is the angle between a tangent vector and $\partial/\partial x_1$, then $(x,\theta)$ are local coordinates in $SM$. In these coordinates, the elements in the Fourier expansion of $f = f(x,\theta)$ are given by 
$$
f_k(x,\theta) = \left( \frac{1}{2\pi} \int_0^{2\pi} f(x,\theta') e^{-ik\theta'} \,d\theta' \right) e^{ik\theta}.
$$
The $H^1$-norm of a function $u\in C^{\infty}(SM)$ will be defined as:
\[\norm{u}^2_{H^{1}}:=\norm{Xu}^2+\norm{X_{\perp}u}^2+\norm{Vu}^2+\norm{u}^2.\]
There is a canonical Riemannian metric on $SM$, called the Sasaki metric, which is  defined by declaring the frame
$\{X,X_{\perp},V\}$ to be an orthonormal basis. If we consider the gradient $\nabla u$ with respect to the Sasaki metric, then the
$H^1$-norm has the familiar form
\[\norm{u}^2_{H^{1}}=\norm{\nabla u}^2+\norm{u}^2.\]

We will make repeated use of the following fundamental $L^2$-energy identity (or Pestov identity) valid for any $u \in C^{\infty}(SM)$ (see \cite{PSU_tensor} for a short proof):
\begin{equation} \label{energy_identity}
\norm{XVu}^2 - (KVu,Vu)+\norm{Xu}^2 - \norm{VXu}^2 = 0.
\end{equation}

We also make use of the splitting $X = \eta_+ + \eta_-$ where 
$$
\eta_+ = \frac{1}{2}(X+i X_{\perp}), \quad \eta_- = \frac{1}{2}(X-i X_{\perp}).
$$
It is easy to check that these operators have the property:
$\eta_+: \Omega_k \to \Omega_{k+1}$ and $\eta_-: \Omega_k \to \Omega_{k-1}$
for any $k\in\Z$ and $\eta_{+}^*=-\eta_{-}$. The operators $\eta_{\pm}$ are elliptic
\cite{GK} and, as seen in the proof below, they are essentially $\dbar$ and $\partial$ operators.

\begin{Lemma} Assume $(M,g)$ has genus $\tt{g} \geq 2$.
Then $\eta_{+}:\Omega_{k}\to \Omega_{k+1}$ is injective for
$k\geq 1$ and $\eta_{-}:\Omega_{k}\to\Omega_{k-1}$ is injective
for $k\leq -1$. The dimension of $\mbox{\rm Ker}\,\eta_{-}$ is $(2k-1)(\tt{g}-1)$ for $k\geq 2$  and $\tt{g}$ for $k=1$. Moreover, $\eta_{-}$ is surjective for $k\geq 2$ and $\eta_{+}$ is surjective for $k\leq -2$.
\label{lemma:conformal}
\end{Lemma}

\begin{proof}Consider isothermal coordinates $(x,y)$ on $M$ such that the metric
can be written as $ds^2=e^{2\lambda}(dx^2+dy^2)$ where $\lambda$ is a smooth
real-valued function of $(x,y)$. This gives coordinates $(x,y,\theta)$ on $SM$ where
$\theta$ is the angle between a unit vector $v$ and $\partial/\partial x$.
In these coordinates, $V=\partial/\partial\theta$ and
the vector fields $X$ and $X_{\perp}$ are given by:
\[X=e^{-\lambda}\left(\cos\theta\frac{\partial}{\partial x}+
\sin\theta\frac{\partial}{\partial y}+
\left(-\frac{\partial \lambda}{\partial x}\sin\theta+\frac{\partial\lambda}{\partial y}\cos\theta\right)\frac{\partial}{\partial \theta}\right);\]
\[X_{\perp}=-e^{-\lambda}\left(-\sin\theta\frac{\partial}{\partial x}+
\cos\theta\frac{\partial}{\partial y}-
\left(\frac{\partial \lambda}{\partial x}\cos\theta+\frac{\partial \lambda}{\partial y}\sin\theta\right)\frac{\partial}{\partial \theta}\right).\]
Consider $u\in\Omega_k$ and write it as $u(x,y,\theta)=h(x,y)e^{ik\theta}$.
Using these formulae a calculation shows that
\begin{equation}
\eta_{-}(u)=e^{-(1+k)\lambda}\dbar(he^{k\lambda})e^{i(k-1)\theta},
\label{eq:eta}
\end{equation}
where $\dbar=\frac{1}{2}\left(\frac{\partial}{\partial x}+i\frac{\partial}{\partial y}\right)$.
For completeness let us write the formula for $\eta_{+}$:
\[\eta_{+}(u)=e^{(k-1)\lambda}{\partial}(he^{-k\lambda})e^{i(k+1)\theta}.\]

Note that $\Omega_k$ can be identified with the set of smooth sections
of the bundle $\kappa^{\otimes k}$ where $\kappa$ is the canonical line bundle. The identification takes $u=he^{ik\theta}$ into $he^{k\lambda}(dz)^k$ ($k\geq 0$)
and $u=he^{-ik\theta}\in \Omega_{-k}$ into $he^{k\lambda}(d\bar{z})^k$.
Hence from (\ref{eq:eta}) we see that for $k\geq 0$, $u$ is in the kernel of $\eta_{-}$
if and only if the corresponding section of $\kappa^{\otimes k}$ is holomorphic.
Hence the dimension of the kernel of $\eta_{-}$ for $k\geq 0$ only depends on the conformal structure of the surface. 
The argument for $\mbox{\rm Ker}\,\eta_{+}$ for $k\leq 0$ is the same.

We can be a bit more precise about the above.
Let $\Gamma(M,\kappa^{\otimes k})$ denote the space of smooth sections of
the $k$-th tensor power of the canonical line bundle $\kappa$. Locally
its elements have the form $w(z)dz^k$ for $k\geq 0$ and $w(z)d\bar{z}^{-k}$ for $k\leq 0$.
Given a metric $g$ on $M$, there is map
\[\varphi_{g}:\Gamma(M,\kappa^{\otimes k})\to \Omega_k\]
given by restriction to $SM$. This map is a complex linear isomorphism.
Let us check what this map looks like in isothermal coordinates.
An element of $\Gamma(M,\kappa^{\otimes k})$ is locally of the form $w(z)dz^k$ ($k\geq 0$).
Consider a tangent vector $\dot{z}=\dot{x}_{1}+i\dot{x}_{2}$. It has norm one in the metric $g$ iff $e^{i\theta}=e^{\lambda}\dot{z}$. Hence the restriction
of $w(z)dz^k$ to $SM$ is
\[w(z)e^{-k\lambda}e^{ik\theta}\]
as indicated above. Observe that $\varphi_{g}$ is surjective because given $u\in\Omega_k$ ($k\geq 0$) we can write it locally
as $u=he^{ik\theta}$ and the local sections $he^{k\lambda}(dz)^k$ glue together to define an element
in $\Gamma(M,\kappa^{\otimes k})$.

Moreover there is also a restriction map
\[\psi_{g}:\Gamma(M,\kappa^{\otimes k}\otimes \bar{\kappa})\to \Omega_{k-1}\]
which is an isomorphism. The restriction of $w(z)dz^k\otimes d\bar{z}$ to $SM$ is
\[w(z)e^{-(k+1)\lambda}e^{i(k-1)\theta},\]
because $e^{-i\theta}=e^{\lambda}\bar{\dot{z}}$.

Given any holomorphic line bundle $\xi$ over $M$, there is a $\dbar$-operator defined on:
\[\dbar:\Gamma(M,\xi)\to \Gamma(M,\xi\otimes \bar{\kappa}).\]
In particular we can take $\xi=\kappa^{\otimes k}$. Combining this with (\ref{eq:eta}) we derive
the following commutative diagram:

\[\begin{CD}
\Gamma(M,\kappa^{\otimes k}) @>\varphi_{g}>> \Omega_{k}\\
@VV\dbar V   @VV\eta_{-}V\\
\Gamma(M,\kappa^{\otimes k}\otimes \bar{\kappa}) @>\psi_{g}>> \Omega_{k-1}
\end{CD}\]
In other words:
\begin{equation}
\eta_{-}=\psi_{g}\,\dbar\,\varphi_{g}^{-1}.
\label{eq:formeta}
\end{equation}
It is well known that on a Riemann surface of genus $\geq 2$, $\dbar$
is surjective for $k\geq 2$ (see for example \cite{Du}) and the dimension of its kernel can be computed by Riemann-Roch if $k \geq 1$. By (\ref{eq:formeta})
$\eta_{-}$ is surjective for $k\geq 2$ and any metric. The result for $\eta_{+}$ follows
in a similar way (or we could use that $\eta_{+}^*=-\eta_{-}$).
\end{proof}

For example for $k=2$, the elements in $\mbox{\rm Ker}\,\eta_{-}$ are in
1-1 correspondence with holomorphic quadratic differentials.
From the lemma we see that given $u\in\Omega_{k}$ ($k\geq 1$), there is
a unique smooth function $v\in\Omega_{k+1}$ orthogonal to $\mbox{\rm Ker}\,\eta_{-}$ such that $\eta_{-}(v)=u$.

Using this lemma we can define ``ladder'' operators as in \cite{GK1}
as follows. Given $f_{r}\in\Omega_{r}$, $r\geq 0$, define a sequence
of functions $f_{r+2},f_{r+4},\cdots,f_{r+2n}$ by requiring:
\[\eta_{+}(f_{r+2i-2})+\eta_{-}(f_{r+2i})=0\;\;\mbox{\rm for}\;\;1\leq i\leq n.\]
The functions $f_{r+2i}$ are uniquely determined by demanding them
to be orthogonal to the kernel of $\eta_{-}:\Omega_{r+2i}\to\Omega_{r+2i-1}$.
Now define $T_{n}:\Omega_{r}\to\Omega_{r+2n}$ by setting $T_{n}(f_{r})=f_{r+2n}$.
If we assume that the Gaussian curvature of the surface is negative, then
it is possible to show that there is good control on the various Sobolev norms
of $T_{n}$ \cite{GK1}. Using the operators $T_n$, Guillemin and Kazhdan prove the
existence of invariant distributions as in Theorems \ref{thm_main_i0star}--\ref{thm_main_imstar}.
Unfortunately these estimates are not available in
the general Anosov case, so we need to proceed in a different manner. We derive our estimates
from the Pestov identity (\ref{energy_identity}).

\section{$\alpha$-controlled surfaces} \label{sec_control}
The following definition is motivated by the Pestov identity (\ref{energy_identity}) and it will
be technically very useful in what follows.
\begin{Definition} Let $\alpha\in [0,1]$. We say that a closed surface $(M,g)$ is $\alpha$-controlled if
\[\|X\psi\|^2-(K\psi,\psi)\geq \alpha \|X\psi\|^2\]
for all $\psi\in C^{\infty}(SM)$.
\end{Definition}

Obviously a surface of non-positive curvature is $1$-controlled. The converse is also true: if a surface is $1$-controlled
then $K\leq 0$ since $(K\psi,\psi)\leq 0$ must hold for any $\psi$. The objective of this section
is to prove the following theorem:

\begin{Theorem} Let $(M,g)$ be a closed surface.
\begin{enumerate}
\item If $(M,g)$ is free of conjugate points, then it is $0$-controlled.
\item If $(M,g)$ is free of focal points, then it is $1/2$-controlled.
\item If $(M,g)$ is Anosov, then it is $\alpha$-controlled for some $\alpha>0$. Moreover, the following stronger result
holds:
\[\|X\psi\|^2-(K\psi,\psi)\geq \alpha (\|X\psi\|^2+\norm{\psi}^2)\]
for all $\psi\in C^{\infty}(SM)$.
\end{enumerate}
\label{theorem:control}
\end{Theorem}

\begin{proof} If $(M,g)$ has no conjugate points, a well known result due to E. Hopf \cite{H} gives the existence
of a bounded measurable function $r:SM\to\re$ such that $r$ is differentiable along the geodesic flow
and satisfies the Riccati equation:
\begin{equation}
Xr+r^2+K=0.
\label{eq:riccati}
\end{equation}
Let $a:SM\to\re$ be any bounded measurable function differentiable along the geodesic flow and let us compute
\begin{equation*}
\begin{aligned}
|X\psi-a\psi|^2&=|X\psi|^2-2\Re(a(X\psi) \bar{\psi})+a^2|\psi|^2\\
&=|X\psi|^2+|\psi|^2(Xa+a^2)-X(a|\psi|^2)
\end{aligned}
\end{equation*}
Integrating this equality over $SM$ and using that the volume form $d(SM)$ is invariant under the
geodesic flow we obtain
\begin{equation}
\norm{X\psi-a\psi}^2=\norm{X\psi}^2+(Xa+a^2,|\psi|^2).
\label{eq:general}
\end{equation}
We now make use of the fact that $a=r$ satisfies the Riccati equation to obtain:
\begin{equation}
\norm{X\psi-r\psi}^2=\norm{X\psi}^2-(K\psi,\psi).
\label{eq:usingr}
\end{equation}
This clearly shows item $(1)$.  In fact, Hopf in \cite{H} shows the existence of two
bounded measurable solutions to (\ref{eq:riccati}) which we call $r^{+}$ and $r^{-}$; they are related by $r^{+}(x,v)=-r^{-}(x,v)$.
From the construction
of these functions it is immediate that if $(M,g)$ is free of focal points then
$r^{+}\geq 0$ and $r^{-}\leq 0$ (compare with \cite{SU}).
 Let $a:=r^{+}+r^{-}$. A simple calculation shows that $a$ satisfies
\[Xa+a^2+2K=2r^{+}r^{-}\leq 0.\]
Using this function $a$ in equality (\ref{eq:general}) we derive
\[\norm{X\psi-a\psi}^2\leq \norm{X\psi}^2-2(K\psi,\psi)\]
which proves item $(2)$.

To prove item $(3)$ we shall exploit the fact that in the Anosov case we
have two continuous (in fact $C^1$) solutions $r^{+},r^{-}$ of the Riccati equation with $r^{+}-r^{-}>0$ everywhere. In \cite{Ebe}, Eberlein shows that a surface with no conjugate points is Anosov
if and only if the limit solutions $r^{+}$ and $r^{-}$ constructed by Hopf are distinct everywhere (later on in Section 7 we will generalize this result for the case of the $\beta$-Jacobi equation).
If this happens then $-X_{\perp}+r^{+,-}V$ spans the bundle $E^{s,u}$. Since the latter is
known to be of class $C^1$ for a surface \cite{HPS}, it follows that in the Anosov case, $r^{+}$ and $r^{-}$
are $C^1$.

Let $A:=X\psi-r^{-}\psi$ and $B:=X\psi-r^{+}\psi$. Using equation (\ref{eq:usingr}) we see that $\|A\|=\|B\|$.
Solving for $\psi$ and $X\psi$ we obtain
\begin{align*}
\psi&=(r^{+}-r^{-})^{-1}(A-B)\\
X\psi&=\lambda A+(1-\lambda)B,
\end{align*}
where
\[\lambda:=\frac{r^{+}}{r^{+}-r^{-}}.\]
From these equations it follows that there exists a constant $\alpha>0$ such that
\[2\alpha\norm{\psi}^2\leq \norm{A}^2,\]
\[2\alpha\norm{X\psi}^2\leq \norm{A}^2\]
and item $(3)$ is proved.

\end{proof}

\begin{Remark}{\rm The proof above shows the following general statement: if there exists
a bounded measurable function $a:SM\to\re$ such that
\[Xa+a^2+\beta K\leq 0\]
then the surface is $(\beta-1)/\beta$-controlled.}
\label{remark:linkbeta}
\end{Remark}

\section{Surjectivity of $I_{0}^*$}
\label{surjectiveadjoint}

In this section we will prove Theorem \ref{thm_main_i0star} in the introduction. The strategy is to deduce properties of the ray transform $I_0$ from properties of the operator $P = VX$ as in \cite{PSU_tensor}. The following result characterizes the injectivity of $I_0$ in terms of $P$.

\begin{Lemma}
Suppose $(M,g)$ has Anosov geodesic flow. The map $I_0: C^{\infty}(M) \to \mathrm{Maps}(\mathcal{G},\re)$ is injective if and only if the only solutions $u \in C^{\infty}(SM)$ of $Pu = 0$ in $SM$ are the constants.
\end{Lemma}
\begin{proof} This follows immediately from the ergodicity of an Anosov flow and the Livsic theorem \cite{dLMM}.
\end{proof}

The next inequalities express the uniqueness properties of $P$ under various assumptions. If $E$ is a subspace of $\mDp(SM)$, we write $E_{\diamond}$ for the subspace of those $v \in E$ with $\langle v, 1 \rangle = 0$.

\begin{Lemma} \label{lemma_p_inequalities}
Let $(M,g)$ be a closed surface.
\begin{enumerate}
\item[(a)]
If $(M,g)$ has no conjugate points, then 
$$
\norm{Xu}\leq \norm{Pu}, \quad u \in C^{\infty}(SM).
$$
\item[(b)]
If $(M,g)$ is Anosov, then there is a constant $C$ such that 
$$
\norm{u}_{H^1} \leq C \norm{Pu}, \quad u \in C^{\infty}_{\diamond}(SM).
$$

\end{enumerate}
\end{Lemma}

\begin{proof}
Item (a) follows from the energy identity (\ref{energy_identity}). The identity reads 
$$
\norm{Pu}^2 = \norm{XVu}^2 - (KVu,Vu) + \norm{Xu}^2, \quad u \in C^{\infty}(SM).
$$
On a surface with no conjugate points, one has by item (1) in Theorem \ref{theorem:control}, $\norm{XVu}^2 - (KVu,Vu) \geq 0$ for any $u \in C^{\infty}(SM)$. This proves (a).

To prove (b) we use the identity above together with
item (3) in Theorem \ref{theorem:control} to derive:
\[\norm{Pu}^2\geq \norm{Xu}^2+\alpha(\norm{Vu}^2+\norm{XVu}^2).\]
Using that $X_{\perp}u=XVu-VXu=XVu-Pu$ we also obtain
\[\norm{X_{\perp}u}^2\leq 2(\norm{XVu}^2+\norm{Pu}^2)\]
and hence there is a constant $C'$ for which
\[C'\norm{Pu}^2\geq \norm{X_{\perp}u}^2+\norm{Vu}^2+\norm{Xu}^2.\]
By the Poincar\'e inequality for closed Riemannian manifolds, there
is another constant $D$ such that
\[\|u\|^2\leq D(\|Xu\|^2+\|X_{\perp}u\|^2+\|Vu\|^2)\]
for all $u\in C^{\infty}_{\diamond}(SM)$ and hence there is a constant $C$ such that
\[\norm{u}_{H^1} \leq C \norm{Pu}\]
for all $u\in C^{\infty}_{\diamond}(SM)$ as desired.
\end{proof}

We now convert the previous uniqueness result for $P$ into a solvability result for $P^* = XV$.

\begin{Lemma} \label{lemma_padjoint_surjective}
Let $(M,g)$ be an Anosov surface. For any $f \in H^{-1}_{\diamond}(SM)$ there is a solution $h\in L^2(SM)$ of the equation 
$$
P^* h = f \quad \text{in } SM.
$$
Further, $\norm{h}_{L^2} \leq C \norm{f}_{H^{-1}}$ with $C$ independent of $f$.
\end{Lemma}
\begin{proof}
Consider the subspace $P C^{\infty}_{\diamond}(SM)$ of $L^2(SM)$. Any element $v$ in this subspace has a unique representation as $v = Pu$ for some $u \in C^{\infty}_{\diamond}(SM)$ by Lemma \ref{lemma_p_inequalities}. Given $f \in H^{-1}_{\diamond}(SM)$, define the linear functional 
$$
l: P C^{\infty}_{\diamond}(SM) \to \C, \ \ l(Pu) = \langle u, f \rangle.
$$
This functional satisfies by Lemma \ref{lemma_p_inequalities} 
$$
\abs{l(Pu)} \leq \norm{f}_{H^{-1}} \norm{u}_{H^1} \leq C \norm{f}_{H^{-1}} \norm{Pu}_{L^2}.
$$
Thus $l$ is continuous on $P C^{\infty}_{\diamond}(SM)$, and by the Hahn-Banach theorem it has a continuous extension 
$$
\bar{l}: L^2(SM) \to \C, \quad \abs{\bar{l}(v)} \leq C \norm{f}_{H^{-1}} \norm{v}_{L^2}.
$$
By the Riesz representation theorem, there is $h \in L^2(SM)$ with 
$$
\bar{l}(v) = (v,h)_{L^2(SM)}, \quad \norm{h}_{L^2} \leq C \norm{f}_{H^{-1}}.
$$
If $u \in C^{\infty}_{\diamond}(SM)$, we have 
\[\langle u, P^*h \rangle = \langle Pu, h \rangle = l(Pu) = \langle u,f \rangle\]
and since $f$ is orthogonal to constants it follows that $P^*h=f$.
\end{proof}

We can now prove surjectivity of $I_0^*$.

\begin{proof}[Proof of Theorem \ref{thm_main_i0star}]
Given $f \in C^{\infty}(M)$, we use Lemma \ref{lemma_padjoint_surjective} to find $h \in L^2(SM)$ satisfying 
$$
P^* h = -Xf.
$$
Define $w = Vh + f$. Then 
$$
Xw = XVh + Xf = P^* h + Xf = 0
$$
and $w_0 = f$ as required. In order to show that $w_{2j}$ are smooth observe that
$Xw=0$ means that $\eta_{+}w_{k-1}+\eta_{-}w_{k+1}=0$. Hence 
$\eta_{-}w_2=-\eta_{+}w_0=-\eta_{+}f$. Since the operators $\eta_{\pm}$ are elliptic
and $f$ is smooth it follows that $w_2$ is smooth. Inductively, we obtain that $w_{2j}$
is smooth for every $j$.
\end{proof}

In fact, the surjectivity of $P^*$ easily implies a more general form of Theorem \ref{thm_main_i0star}.

\begin{Theorem} \label{lemma_surjectivity_stronger}
Let $g \in H^{-1}_{\diamond}(SM)$ and $f \in L^2(M)$. There exists $w \in H^{-1}(SM)$ satisfying $Xw = g$ in $SM$ and $w_0 = f$.
\end{Theorem}
\begin{proof}
By Lemma \ref{lemma_padjoint_surjective} there is $h \in L^2(SM)$ with 
$$
P^* h = g - Xf. 
$$
Then $w = Vh + f \in H^{-1}(SM)$ satisfies 
$$
Xw = XVh + Xf = g
$$
and $w_0 = f$.
\end{proof}

\section{Surjectivity of $I^*_{m}$ for $m\geq 1$} \label{sec_surjective2}

In this section we prove Theorem \ref{thm_main_i1star}  and we pave the way for the proof of Theorem \ref{thm_main_imstar}.  Fix $m \geq 1$, and let $T:C^{\infty}(SM)\to \bigoplus_{|k|\geq m+1}\Omega_k$ be the projection operator 
\[Tu=\sum_{|k|\geq m+1}u_k.\]
In other words $T$ is defined by $u=\sum_{|k|\leq m}u_k+Tu$.
Now let $Q:=TVX=TP$, clearly $Q^*=XVT$, since $T$ is self-adjoint. Directly from the definitions
we have
\begin{equation}
\| Pu\|^2=\sum_{|k|\leq m}k^2\| (Xu)_{k}\|^2+\| Qu\|^2.
\label{eq:Q1}
\end{equation}

\begin{Lemma} Let $(M,g)$ be an Anosov surface. Assume there exists a constant $C$ such that
\[\|Xu\|\leq C\|Qu\|\]
for any $u\in \bigoplus_{|k|\geq m}\Omega_k$. Then there exists another
constant $D$ such that
\[\|u\|_{H^{1}}\leq D\|Qu\|\]
for any $u\in  \bigoplus_{|k|\geq m}\Omega_k$.
\label{lemma:simple}
\end{Lemma}

\begin{proof} Using equation (\ref{eq:Q1}) we see that there is a constant $c$ depending on $m$ such that
\[\|Pu\|^2\leq c\|Xu\|^2+\|Qu\|^2\]
and therefore using the hypothesis we derive the existence  of a constant $C'$ such that
\[\|Pu\|\leq C'\|Qu\|\]
for any $u\in  \bigoplus_{|k|\geq m}\Omega_k$. The result now follows from Lemma \ref{lemma_p_inequalities}.

\end{proof}

This simple lemma indicates that in order to obtain sub-elliptic estimates for the operator $Q$ we must investigate when
there exists a constant $C$ such that
\begin{equation*}
\|Xu\|\leq C\|Qu\|
\end{equation*}
for any $u\in \bigoplus_{|k|\geq m}\Omega_k$.  Certainly this estimate implies solenoidal injectivity of $I_{m}$: indeed suppose
$Xv=f$, where $f$ has degree $m$ and let $u=v-\sum_{|k|\leq m-1}v_k$. Then $Xu$ has degree $m$ and 
$Qu=TVXu=0$. Since $u\in \bigoplus_{|k|\geq m}\Omega_k$, we deduce that $Xu=0$ and hence $u=0$ which in turn implies
that $v$ has degree $m-1$ as required by $s$-injectivity.
The next proposition will be very useful for our purposes.

\begin{Proposition} Let $(M,g)$ be a closed surface which is $\alpha$-controlled and let $m$ be an integer $\geq 1$. Then given any $u\in \bigoplus_{|k|\geq m}\Omega_k$ we have
\begin{equation}
\begin{aligned}
\|Qu\|^2&\geq (1-m^2+\alpha(m+1)^2)(\|\eta_{-}u_{m+1}\|^2+\|\eta_{+}u_{-m-1}\|^2)\\
&+(1-(m-1)^2+\alpha m^2)(\|\eta_{-}u_{m}\|^2+\|\eta_{+}u_{-m}\|^2)+\alpha\|w\|^2+\|v\|^2
\end{aligned}
\label{eq:quantitative}
\end{equation}
where $v:=\sum_{|k|\geq m+1}(Xu)_k$ and $w:=\sum_{|k|\geq m+1}(XVu)_{k}$.

\label{prop:alphac}
\end{Proposition}

\begin{proof} Recall that $X=\eta_{+}+\eta_{-}$. First note that if $u\in \bigoplus_{|k|\geq m}\Omega_k$ then
\[\|Xu\|^2=\|\eta_{-}u_{m+1}\|^2+\|\eta_{+}u_{-m-1}\|^2+\|\eta_{-}u_{m}\|^2+\|\eta_{+}u_{-m}\|^2+\|v\|^2,\]
where $v=\sum_{|k|\geq m+1}v_k=\sum_{|k|\geq m+1}(Xu)_k$. Note now that (\ref{eq:Q1}) may be written as
\[\|Pu\|^2=\|Qu\|^2+m^2(\|\eta_{-}u_{m+1}\|^2+\|\eta_{+}u_{-m-1}\|^2)+(m-1)^2(\|\eta_{-}u_{m}\|^2+\|\eta_{+}u_{-m}\|^2).\]
A similar calculation shows that
\[\|XVu\|^2=(m+1)^2(\|\eta_{-}u_{m+1}\|^2+\|\eta_{+}u_{-m-1}\|^2)+m^2(\|\eta_{-}u_{m}\|^2+\|\eta_{+}u_{-m}\|^2)+\|w\|^2,\]
where $w=\sum_{|k|\geq m+1}w_k=\sum_{|k|\geq m+1}(XVu)_{k}$. 

We make use of the key energy identity (\ref{energy_identity}):
\[\|Pu\|^2=\|XVu\|^2-(KVu,Vu)+\|Xu\|^2\]
and use the hypotheses to deduce
\[\|Pu\|^2\geq \alpha\|XVu\|^2+\|Xu\|^2.\]
Making the appropriate substitutions we obtain:
\begin{equation*}
\begin{aligned}
\|Qu\|^2&\geq (1-m^2+\alpha(m+1)^2)(\|\eta_{-}u_{m+1}\|^2+\|\eta_{+}u_{-m-1}\|^2)\\
&+(1-(m-1)^2+\alpha m^2)(\|\eta_{-}u_{m}\|^2+\|\eta_{+}u_{-m}\|^2)+\alpha\|w\|^2+\|v\|^2
\end{aligned}
\end{equation*}
as desired.

\end{proof}

\begin{Corollary} Let $(M,g)$ be an Anosov surface which is $\alpha$-controlled for
$\alpha>(m-1)/(m+1)$. Then there exists a constant $C$ such that
\[\|u\|_{H^{1}}\leq C\|Qu\|\]
for any $u\in  \bigoplus_{|k|\geq m}\Omega_k$.
\label{cor:subellipticestimate}
\end{Corollary}

\begin{proof} From Proposition \ref{prop:alphac}  we see that if $\alpha>(m-1)/(m+1)$, then there is a positive constant $C$ such that
\[\|Qu\|\geq C\|Xu\|.\]
We can now use Lemma \ref{lemma:simple} to prove the corollary.
\end{proof}

\begin{Lemma} Let $(M,g)$ be an Anosov surface which is $\alpha$-controlled for $\alpha>(m-1)/(m+1)$.
Then given $f\in H^{-1}(SM)$ with $f_k=0$ for $|k|\leq m-1$, there exists $h\in L^{2}(SM)$ such that
\[Q^*h=f.\]
Further, $\|h\|_{L^2}\leq C\|f\|_{H^{-1}}$ for a constant $C$ independent of $f$.
\label{lemma:solution}
\end{Lemma}

\begin{proof} The proof is quite similar to that of Lemma \ref{lemma_padjoint_surjective}.
Consider the subspace $Q \bigoplus_{|k|\geq m}\Omega_k$ of $L^2(SM)$. Any element $v$ in this subspace has a unique representation as $v = Qu$ for some $u \in \bigoplus_{|k|\geq m}\Omega_k$ by Corollary \ref{cor:subellipticestimate}. Given $f$ as in the statement of the lemma, define the linear functional 
$$
l: Q \bigoplus_{|k|\geq m}\Omega_k \to \C, \ \ l(Qu) = \langle u, f \rangle.
$$
This functional satisfies by Corollary \ref{cor:subellipticestimate}
$$
\abs{l(Qu)} \leq \norm{f}_{H^{-1}} \norm{u}_{H^1} \leq C \norm{f}_{H^{-1}} \norm{Qu}_{L^2}.
$$
Thus $l$ is continuous on $Q \bigoplus_{|k|\geq m}\Omega_k$, and by the Hahn-Banach theorem it has a continuous extension 
$$
\bar{l}: L^2(SM) \to \C, \quad \abs{\bar{l}(v)} \leq C \norm{f}_{H^{-1}} \norm{v}_{L^2}.
$$
By the Riesz representation theorem, there is $h \in L^2(SM)$ with 
$$
\bar{l}(v) = (v,h)_{L^2(SM)}, \quad \norm{h}_{L^2} \leq C \norm{f}_{H^{-1}}.
$$
If $u \in C^{\infty}(SM)$, we have 
\begin{align*}
\langle u, Q^*h \rangle &= \langle Qu, h \rangle = \langle Q(u - \sum_{|k|\leq m-1} u_k),h \rangle = l( Q(u - \sum_{|k|\leq m-1} u_k)) \\
 &= \langle  u - \sum_{|k|\leq m-1} u_k,f \rangle = \langle u,f \rangle,
\end{align*}
where the last equality holds because $f_k=0$ for all $k$ with $|k|\leq m-1$.

\end{proof}

\begin{Theorem}[Surjectivity of $I_{1}^*$] Let $(M,g)$ be an Anosov surface. Suppose $a_{-1}+a_{1}\in\Omega_{-1}\oplus \Omega_{1}$
satisfies $\eta_{+}a_{-1}+\eta_{-}a_{1}=0$. Then there exists $w\in H^{-1}(SM)$ such that
$Xw=0$ and $w_{-1}+w_{1}=a_{-1}+a_{1}$.
\label{theorem:sur1}
\end{Theorem}

\begin{proof} On account of Theorem \ref{theorem:control}, we know that any Anosov surface is $\alpha$-controlled for some $\alpha>0$, hence the hypotheses of
Lemma \ref{lemma:solution} are satisfied for $m=1$. Let $f:=-X(a_{-1}+a_{1})$ and note that
$\eta_{+}a_{-1}+\eta_{-}a_{1}=0$ is equivalent to saying that $f_0=0$. Thus by Lemma \ref{lemma:solution}
there is a function $h\in L^2(SM)$ such that
\[Q^*h=XVTh=-X(a_{-1}+a_{1}).\]
If we let $w:=VTh+a_{-1}+a_{1}$, then $Xw=0$ and $w_{-1}+w_{1}=a_{-1}+a_{1}$.
\end{proof}

Actually, the proof shows that $w=V T h+a_{-1}+a_{1}$ where $h\in L^2(SM)$, $\norm{h}_{L^2} \leq C \norm{a_{-1} + a_1}_{L^2}$. This result easily implies the following which will be the main tool in the proof of Theorem \ref{thm:I2}.  We use the mixed norm spaces 
$$
L^2_x H^{s}_{\theta}(SM) = \{ u \in \mDp(SM) \,;\, \norm{u}_{L^2_x H^s_{\theta}} < \infty \}, \ \ \norm{u}_{L^2_x H^s_{\theta}} = \left( \sum_{k=-\infty}^{\infty} \langle k \rangle^{2s }\norm{u_k}_{L^2}^2 \right)^{1/2},
$$
where as usual $\langle k \rangle=(1+k^2)^{1/2}$.

\begin{Theorem} Let $(M,g)$ be an Anosov surface. Suppose $a_{1}\in \Omega_{1}$ and $\eta_{-}a_{1}=0$. Then there exists $w = \sum_{k=1}^{\infty} w_k \in L^2_x H^{-1}_{\theta}(SM)$ such that $Xw=0$, $w_{1}=a_{1}$, each $w_k$ is in $C^{\infty}(SM)$, and 
$$
\norm{w}_{L^2_x H^{-1}_{\theta}} \leq C \norm{a_1}_{L^2}.
$$
\label{theorem:sur1'}
\end{Theorem}
\begin{proof}
Let $\tilde{w}$ be the distribution given by Theorem \ref{theorem:sur1} in the case where $a_{-1} = 0$, and let $w$ be its holomorphic projection, $w = \sum_{k=1}^{\infty} \tilde{w}_k$. It is easy to check that $(Xw)_k = \eta_+ w_{k-1} + \eta_- w_{k+1} = 0$ for all $k$, so $Xw = 0$. The fact that each $w_k$ is $C^{\infty}$ follows by elliptic regularity from the equations for $(Xw)_k$. Finally, since $w = V(\sum_{k=2}^{\infty} h_k) + a_1$ with $\norm{h}_{L^2} \leq C \norm{a_1}_{L^2}$ we obtain the norm estimate.
\end{proof}

\medskip

\begin{proof}[Proof of Theorem \ref{thm_main_i1star}]
Theorem \ref{thm_main_i1star} follows from Theorem \ref{theorem:sur1} if we prove the following: let $A=a_{-1}+a_{1}$ be a 1-form. Then $A$ is solenoidal if and only if 
$\eta_{+}a_{-1}+\eta_{-}a_{1}=0$. Note that the claim about smoothness of $w_{k}$ for $k$ odd
follows as in the proof of Theorem \ref{thm_main_i0star} using the ellipticity of $\eta_{\pm}$.

The 1-form $A$ is solenoidal if and only if $d\star A=0$, where $\star$ is the Hodge star operator
of the metric $g$. Let $j$ denote the complex structure of $(M,g)$.
It is easy to check that for any 1-form $\beta$ we have
\[d\beta_{x}(v,jv)=(X_{\perp}(\beta)-X(\star\beta))(x,v).\]
where $(x,v)\in SM$. Hence $d\star A=0$ if and only if
\[X_{\perp}(\star A)+X(A)=0.\]
But if $A=a_{-1}+a_{1}$, then $\star A=ia_{-1}-ia_{1}$ and thus the previous equation turns into
\[iX_{\perp}a_{-1}-iX_{\perp}a_{1}+Xa_{-1}+Xa_{1}=0\]
or equivalently
\[\eta_{+}a_{-1}+\eta_{-}a_{1}=0\]
as desired.
\end{proof}

\begin{Theorem}[Surjectivity of $I_{m}^*$ for $m\geq 2$] Let $(M,g)$ be an Anosov surface which is $\alpha$-controlled for $\alpha>(m-1)/(m+1)$
and $m\geq 2$. Let $q_{m}\in \Omega_m$ be such that $\eta_{-}q_{m}=0$. Then there exists $w\in H^{-1}(SM)$ such that
$Xw=0$ and $w_m=q_m$.
\label{theorem:sur2}
\end{Theorem}

\begin{proof} Let $f:=-Xq_m$. By hypothesis, $f_k=0$ for all $k \neq m+1$. By Lemma \ref{lemma:solution} there is $h\in L^2(SM)$ such that
$XVTh=-Xq_m$. Hence $w=VTh+q_m$ is the desired distribution.
\end{proof}

\section{Injectivity of $I_{m}$} \label{sec_injectivity}

In this section we prove Theorem \ref{thm_main_transport} which is in turn a consequence of a more general
result.

\begin{Theorem} Let $(M,g)$ be a closed surface of genus $\geq 2$ which is $(m-1)/(m+1)$-controlled. Let $f$ be any symmetric $m$-tensor and assume there exists a smooth solution $a$ to the transport equation
\[Xa=f.\]
Then $a_k=0$ for $|k|\geq m$ and $f$ is potential.
\label{theorem:stran}
\end{Theorem}

\begin{proof} Let $u=a-\sum_{|k|\leq m-1}a_k$. Then $Xu$ has degree $m$ and 
$Qu=TVXu=0$. Let us apply inequality (\ref{eq:quantitative}) for $\alpha=(m-1)/(m+1)$
to obtain that 
\begin{equation*}
\begin{aligned}
&XVu=i(m+1)\eta_{-}u_{m+1}-i(m+1)\eta_{+}u_{-m-1},\\
&Xu=\eta_{-}u_{m+1}+\eta_{+}u_{-m-1}.
\end{aligned}
\end{equation*}
Using that $X_{\perp}=XV-VX$ we also obtain
\[X_{\perp}u=i\eta_{-}u_{m+1}-i\eta_{+}u_{-m-1}.\]
Thus
\[\eta_{+}u=\eta_{+}u_{-m-1}\in\Omega_{-m},\]
\[\eta_{-}u=\eta_{-}u_{m+1}\in\Omega_{m}.\]
Since $u_k = 0$ for $\abs{k} < m$, we obtain $\eta_+ u_k = 0$ for $k \neq -m-1$ and $\eta_- u_k = 0$ for $k \neq m+1$. But from Lemma \ref{lemma:conformal} we know that the operator $\eta_{+}$ is injective on $\Omega_k$ for $k\geq 1$ and $\eta_{-}$ is injective
on $\Omega_k$ for $k\leq -1$. This readily implies  $u=0$ and thus $a$ must have degree $m-1$.
This also implies easily that $f$ is a potential tensor (see for example \cite{PSU_tensor}).
\end{proof}

\begin{proof}[Proof of Theorem \ref{thm_main_transport}.]
This is now a direct consequence of the previous theorem and Theorem \ref{theorem:control}.
\end{proof}

\section{$SL(2,\re)$-cocycles, Hopf solutions and terminator values of surfaces} \label{sec_betaconjugate}

Let $(M,g)$ be a closed oriented Riemannian surface. The usual Jacobi equation
$\ddot{y}+K(t)y=0$ determines the differential of the geodesic flow $\phi_t$: if we
fix $(x,v)\in SM$ and $T_{(x,v)}(SM)\ni \xi=-a X_{\perp}+bV$ then
\[d\phi_{t}(\xi)= -y(t)X_{\perp}(\phi_{t}(x,v))+\dot{y}(t)V(\phi_{t}(x,v)),\]
where $y(t)$ is the unique solution to the Jacobi equation with initial
conditions $y(0)=a$ and $\dot{y}(0)=b$ and $K(t)=K(\pi\circ\phi_{t}(x,v))$.
 The differential of the geodesic flow
determines an $SL(2,\re)$-cocyle over $\phi_t$ with infinitesimal generator:
\[A:=\left(
  \begin{array}{ c c }
     0 & 1 \\
     -K & 0
  \end{array} \right).\]
Given a real number $\beta$ we consider the following 1-parameter family of infinitesimal
generators:
\[A_{\beta}:=\left(
  \begin{array}{ c c }
     0 & 1 \\
     -\beta K & 0
  \end{array} \right).\]
They determine by integration a 1-parameter family of $SL(2,\re)$-cocycles $\Psi^{\beta}_{t}$ over the
geodesic flow (see \cite{Ka} for information on cocycles over dynamical systems). More precisely, $\Psi_t^{\beta}$ is the matrix given by 
$$
\Psi_t^{\beta}(x,v): \left( \begin{array}{c} y(0) \\ \dot{y}(0) \end{array} \right) \mapsto \left( \begin{array}{c} y(t) \\ \dot{y}(t) \end{array} \right)
$$
where $\ddot{y}(t) + \beta K(\pi \circ \phi_t(x,v)) y(t) = 0$. Since $A_{\beta}$ has trace zero, $\Psi_{t}^{\beta}\in SL(2,\re)$.
Clearly $\Psi^{1}_{t}$ can be identified with $d\phi_t$ acting on the kernel of the contact 1-form of the geodesic flow (i.e. the 2-plane spanned by $X_{\perp}$ and $V$). In this section we shall study this family of cocycles putting emphasis on two properties: absence of conjugate points and hyperbolicity.
For completeness we first give the following two definitions.

\begin{Definition} The cocycle $\Psi_{t}^{\beta}$ is free of conjugate points if any non-trivial
solution of the $\beta$-Jacobi equation $\ddot{y}+\beta K(t)y=0$ with $y(0)=0$ vanishes only at $t=0$.
\end{Definition}

\begin{Definition} The cocycle $\Psi_{t}^{\beta}$ is said to be hyperbolic if there
is a continuous invariant splitting
$\re^2=E^{u}\oplus E^{s}$, and constants $C>0$ and $0<\rho<1<\eta$ such that 
for all $t>0$ we have
\[\|\Psi^{\beta}_{-t}|_{E^{u}}\|\leq C\,\eta^{-t}\;\;\;\;\mbox{\rm
and}\;\;\;\|\Psi^{\beta}_{t}|_{E^{s}}\|\leq C\,\rho^{t}.\]
Note that $E^{s}$ and $E^{u}$ are 1-dimensional subbundles over $SM$.
\end{Definition}

Of course, saying that $\Psi^{1}_{t}$ is hyperbolic is the same as saying that $(M,g)$ is an Anosov surface. The two properties are related by the following:

\begin{Theorem} If $\Psi^{\beta}_{t}$ is hyperbolic then $E^s$ and $E^u$ are transversal
to the line generated by $(0,1)$ and $\Psi^{\beta}_{t}$ is free of conjugate points.
\label{thm:kling}
\end{Theorem}

\begin{proof} For $\beta=1$ this is exactly the content of Klingenberg's theorem mentioned
in the introduction \cite{K}. The proof presented in \cite[Chapter 2]{Pa} of this result
extends to the cocycle $\Psi_{t}^{\beta}$ without any significant change.
The key point is that the projectivised action of $\Psi_{t}^{\beta}$ is transversal
to the section given by $(0,1)$.

\end{proof}

Let us describe now the Hopf limit solutions when $\Psi_{t}^{\beta}$ is free of conjugate points \cite{H} (see also Section 1 of \cite{BBB}).
Consider the Riccati equation
\[\dot{r}+r^2+\beta K=0.\]
This equation is obtained from the Jacobi equation $\ddot{y}+\beta Ky=0$ by the change
of variable $r=\dot{y}/y$. The times $t_1<t_2$ are adjacent zeros of a solution
of the Jacobi equation if and only if the corresponding solution $r$ of the Riccati equation
is defined on $(t_1,t_2)$ and $r(t)\to +\infty$ as $t$ decreases to $t_1$ and
$r(t)\to-\infty$ as $t$ increases to $t_2$.

Assume now that $\Psi_{t}^{\beta}$ is free of conjugate points. Then the solutions
$r^{+}_{R}(x,v,t)$ and $r^{-}_{R}(x,v,t)$ of the Riccati equation
$\dot{r}+r^2+\beta K(\pi\circ\phi_{t}(x,v))=0$ with
$r^{+}_{R}(x,v,-R)=+\infty$ and $r^{-}_{R}(x,v,R)=-\infty$ are defined for all $t>-R$ and all $t<R$ respectively.

Consider now a value of $t$ with $|t|<R$. Then $r^{+}_{R}(x,v,t)$ and $r^{-}_{R}(x,v,t)$ are both defined and are decreasing and increasing functions of $R$ respectively. Also $r^{+}_{R}(x,v,t)>r^{-}_{R}(x,v,t)$. Then the limit solutions
\[r^{\pm}(x,v,t):=\lim_{R\to\infty} r^{\pm}_{R}(x,v,t)\]
are defined for all $t$ and $r^{+}\geq r^{-}$. Observe that $r^{+}(x,v,t)$ (resp. $r^{-}(x,v,t)$) is upper (resp. lower) semicontinuous in $(x,v)$. Indeed, if $(x_n,v_n)\to (x,v)$ for each fixed $t$ we have
\[\limsup _{n\to\infty} r^{+}(x_n,v_n,t)\leq \lim_{n\to\infty} r^{+}_{R}(x_n,v_n,t)=r^{+}_{R}(x,v,t).\]
Finally, since $r^{\pm}_{R}(x,v,t+s)=r^{\pm}_{R\pm t}(\phi_{t}(x,v),s)$ it follows that
$r^{\pm}(\phi_{t}(x,v),s)=r^{\pm}(x,v,s+t)$ and hence they define measurable functions
$r^{\pm}:SM\to \re$ solving $Xr+r^2+\beta K=0$.
A simple comparison argument as in \cite{H} shows that $r^{\pm}$ are actually bounded.
We call these functions on $SM$ {\it the Hopf solutions} and often we shall use a subscript $\beta$
to indicate that they are associated with the cocycle $\Psi_{t}^{\beta}$.

\begin{Theorem} Assume that $\Psi_{t}^{\beta}$ is free of conjugate points. Then
$\Psi_{t}^{\beta}$ is hyperbolic if and only if $r_{\beta}^{+}$ and $r_{\beta}^{-}$ are distinct
everywhere.
\label{thm:eberlein}
\end{Theorem}

\begin{proof} For $\beta=1$ this was proved by Eberlein in \cite{Ebe}. To prove the theorem for arbitrary $\beta$ we shall make use of Theorem 0.2
in \cite{CI}. When applied to our situation, it says that
$\Psi_{t}^{\beta}$ is hyperbolic if and only if
\begin{equation}
\sup_{t\in\re}\,\norm{\Psi_{t}^{\beta}(\xi)}=+\infty\;\;\;\mbox{\rm for\;all}\;\xi\in \re^2,\;\xi\neq 0.
\label{eq:quasihyp}
\end{equation}

We shall also need the following proposition:

\begin{Proposition} Assume $\Psi_{t}^{\beta}$ is free of conjugate points
and let $\gamma$ be a unit speed geodesic.
Given $A>0$ there exists $T=T(A,\gamma)$ such that for any solution
$w$ of $\ddot{w}+\beta K(\gamma(t))w=0$ with $w(0)=0$ we have
\[|w(s)|\geq A |\dot{w}(0)|\]
for all $s\geq T$.
\label{prop:eb2.9}
\end{Proposition}

\begin{proof} The proof of this is exactly like the proof of Proposition 2.9 in \cite{Ebe} and hence we omit it.
\end{proof}

Suppose now we have a solution $y$ to the $\beta$-Jacobi equation
$\ddot{y}+\beta Ky=0$ that is bounded in forward time, i.e., there is
$C$ such that $|y(t)|\leq C$ for all $t\geq 0$.
We claim that $r^{-}_{\beta}(x,v,0)y(0)=\dot{y}(0)$. For $R>0$, consider the unique
solution $y_{R}$ of the $\beta$-Jacobi equation with $y_{R}(R)=0$
and $y_{R}(0)=1$. By definition $r^{-}_{R}(x,v,t)=\dot{y}_{R}(t)/y_{R}(t)$.
Let $w(t):=y(t)-y(0)y_{R}(t)$. Since $w(0)=0$ we may apply Proposition \ref{prop:eb2.9} to derive for any $A$, the existence of $T$ such that
\[|w(s)|\geq A |\dot{w}(0)|\]
for all $s\geq T$. Consider $R$ large enough so that $R\geq T$. Then
\[C\geq |y(R)|=|w(R)|\geq A|\dot{w}(0)|\geq A|\dot{y}(0)-r^{-}_{R}(x,v,0)y(0)|.\]
Now let $R\to\infty$ to obtain
\[C\geq A|\dot{y}(0)-r^{-}_{\beta}(x,v,0)y(0)|\]
and since $A$ is arbitrary the claim $r^{-}(x,v,0)y(0)=\dot{y}(0)$ follows.

Similarly, if there is a solution $y$ to the $\beta$-Jacobi equation
that is bounded backwards in time we must have $r^{+}_{\beta}(x,v,0)y(0)=\dot{y}(0)$.
Thus if there is a solution $y$ bounded for all 
times then $r^{+}_{\beta}=r^{-}_{\beta}$ along $\gamma$.

Now it is easy to complete the proof of the theorem. Suppose $\Psi_{t}^{\beta}$ is hyperbolic. Then if we consider a solution of the $\beta$-Jacobi equation corresponding to the stable bundle, it must bounded forward in time by definition
of hyperbolicity and hence by the above $(1,r_{\beta}^{-})$ spans
$E^s$. Similarly $(1,r_{\beta}^{+})$ spans
$E^u$. Since $E^s$ and $E^u$ are transversal $r_{\beta}^{+}$ and $r_{\beta}^{-}$ are distinct everywhere.

Suppose now $r_{\beta}^{+}$ and $r_{\beta}^{-}$ are distinct everywhere.
By the argument above, any non-trivial solution $y$ of the $\beta$-Jacobi equation must be unbounded. Since
\[ \norm{\Psi_{t}^{\beta}(\xi)}^2=y(t)^2+\dot{y}(t)^2,\]
where $y$ is the unique solution to the $\beta$-Jacobi equation with
$(y(0),\dot{y}(0))=\xi$, it follows that (\ref{eq:quasihyp}) holds
and hence $\Psi_{t}^{\beta}$ is hyperbolic.

\end{proof}

Below we will find convenient as in \cite[Section 1]{BBB} to use the following elementary comparison lemma:

\begin{Lemma} Let $r_{i}(t)$, $i=0,1$ be solutions of the initial value problems
\[\dot{r_{i}}+r_{i}^2+K_{i}(t)=0,\;\;r_{i}(0)=w_i,\;\;i=0,1.\]
Suppose $w_1\geq w_0$, $K_{1}(t)\leq K_{0}(t)$ for $t\in [0,t_{0}]$, and $r_{0}(t_0)$ is defined.
Then $r_{1}(t)\geq r_{0}(t)$ for $t\in [0,t_0]$.
\label{lemma:comp}
\end{Lemma}

\begin{Theorem} Let $\beta_0> 0$.
If $\Psi_{t}^{\beta_{0}}$ is free of conjugate points,
then for any $\beta\in [0,\beta_0]$, $\Psi_{t}^{\beta}$ is also free of conjugate points.
If $\Psi_{t}^{\beta_{0}}$ is hyperbolic, then for any $\beta\in (0,\beta_0]$, $\Psi_{t}^{\beta}$ is also hyperbolic.
\end{Theorem}

\begin{proof} Let $r^{\pm}_{\beta_{0}}$ be the Hopf solutions associated with
$\Psi_{t}^{\beta_{0}}$. Given $a\in [0,1]$ we have
\[X(ar^{\pm}_{\beta_{0}})+(ar^{\pm}_{\beta_{0}})^2+a\beta_{0}K=(r^{\pm}_{\beta_0})^2a(a-1)\leq 0.\]
This already implies that the cocycle $\Psi_{t}^{a\beta_0}$ is free of conjugate points.
Indeed, let $q^{\pm}:=a\beta_{0}K-(r^{\pm}_{\beta_{0}})^2a(a-1)$. Then
\[X(ar^{\pm}_{\beta_{0}})+(ar^{\pm}_{\beta_{0}})^2+q^{\pm}=0\]
and $q^{\pm}\geq a\beta_{0}K$. Lemma \ref{lemma:comp} implies that the cocycle $\Psi_t^{a\beta_{0}}$ is free conjugate points. Moreover, it also implies that
\[r^{+}_{a\beta_{0}, R}(x,v,t)\geq ar^{+}_{\beta_{0}}(x,v,t)\]
for all $t>-R$. By letting $R\to\infty$ we derive
\[r^{+}_{a\beta_{0}}\geq ar^{+}_{\beta_{0}}\]
and similary
\[ar^{-}_{\beta_{0}}\geq r^{-}_{a\beta_0}.\]
Putting everything together we have
\begin{equation}
r^{+}_{a\beta_{0}}\geq ar^{+}_{\beta_{0}}\geq ar^{-}_{\beta_{0}}\geq r^{-}_{a\beta_0}.
\label{eq:chainineq}
\end{equation}

Suppose now that $\Psi_{t}^{\beta_0}$ is hyperbolic. Then by Theorem \ref{thm:kling}, $\Psi_{t}^{\beta_0}$ is free of conjugate points and by Theorem \ref{thm:eberlein}
$r^{+}_{\beta_{0}}>r^{-}_{\beta_0}$ everywhere. For $a\in (0,1]$, the chain of inequalities
(\ref{eq:chainineq}) implies that $r^{+}_{a\beta_{0}}>r^{-}_{a\beta_0}$ everywhere
and again by Theorem \ref{thm:eberlein}, $\Psi_{t}^{a\beta_0}$ is hyperbolic.

\end{proof}

This theorem motivates the following definition.

\begin{Definition} Let $(M,g)$ be a closed oriented Riemannian surface.
Let $\beta_{Ter}\in [0,\infty]$ denote the supremum of the values of $\beta\geq 0$ for which
$\Psi_{t}^{\beta}$ is free of conjugate points. We call $\beta_{Ter}$ the terminator
value of the surface.
\end{Definition}

It is easy to check from the definitions that $\Psi_{t}^{\beta_{Ter}}$ is free of conjugate
points.  Indeed if $\Psi_{t}^{\beta_{Ter}}$ has conjugate points, there is a geodesic $\gamma$ and
a non-trivial solution $y(t)$ of the $\beta_{Ter}$-Jacobi equation along $\gamma$ with $y(0)=0$ and $y(a)=0$ for some $a>0$.
Since $\dot{y}(a)\neq 0$ we see that for $\beta$ near $\beta_{Ter}$, the $\beta$-Jacobi equation has conjugate points which
contradicts the definition of $\beta_{Ter}$.

A surface has curvature $K\leq 0$ if and only if $\beta_{Ter}=\infty$.  Indeed, suppose $\beta_{Ter}=\infty$ and there is
a point $x\in M$ with $K(x)>0$. Then $K\geq \delta>0$ for points in a neighbourhood $U$ of $x$. By choosing $\beta$ 
large enough (depending on $\delta$) we can produce $\beta$-conjugate points in $U$ and $\beta_{Ter}<\infty$.
 
If a surface has no focal points, then the argument in the proof of Theorem \ref{theorem:control} shows
that $\beta_{Ter}\geq 2$.

We now have the following purely geometric characterization of hyperbolicity (the parameter $\beta$
is always $\geq 0$ in what follows).

\begin{Theorem} The cocycle $\Psi_{t}^\beta$ is hyperbolic if and only if $\beta\in (0,\beta_{Ter})$ and there is no geodesic trapped in the region of zero Gaussian curvature.
\end{Theorem}

\begin{proof} We know that if $\Psi_{t}^{\beta}$ is hyperbolic then $\beta\leq \beta_{Ter}$.
Since hyperbolicity is an open condition we must have $\beta<\beta_{Ter}$. Finally if there
is a geodesic trapped in zero curvature the cocycle cannot be hyperbolic since
the solutions of $\ddot{y}=0$ have at most linear growth in $t$.

Consider $\beta\in (0,\beta_{Ter})$ and assume that $\Psi_t^{\beta}$ is not
hyperbolic. By Theorem \ref{thm:eberlein} there is a geodesic $\gamma$ along which
$r^{+}_{\beta}=r^{-}_{\beta}$.  Let $a:=\beta/\beta_{Ter}$. Using (\ref{eq:chainineq})
for $\beta_{0}=\beta_{Ter}$ we deduce that along $\gamma$ we must have
\[u:=r^{+}_{\beta}=ar^{+}_{\beta_{Ter}}=ar^{-}_{\beta_{Ter}}= r^{-}_{\beta}.\]
Hence $u$ solves $\dot{u}+u^2+\beta K(\gamma(t))=0$ and 
$\dot{u}/a+(u/a)^2+\beta_{Ter} K(\gamma(t))=0$. It follows that $u^2=u^2/a$ and hence
$u\equiv 0$ and $K(\gamma(t))\equiv 0$ which contradicts our hypotheses.

\end{proof}

As an immediate consequence we obtain the following geometric characterization of Anosov surfaces which was
announced in the introduction.

\begin{Corollary} A closed surface $(M,g)$ is Anosov if and only if there is no geodesic
trapped in the region of zero Gaussian curvature and $\beta_{Ter}>1$.
\label{corollary:ganosov}
\end{Corollary}

We are now in good shape to complete the proofs of Theorems \ref{thm_sinjectivity_main} and \ref{thm_main_imstar}  from the Introduction.

\medskip

\begin{proof}[Proof of Theorem \ref{thm_sinjectivity_main}] By Corollary \ref{corollary:ganosov} the surface is Anosov.
If $\beta_{Ter}\geq (m+1)/2$, the surface is $(m-1)/(m+1)$-controlled by Remark \ref{remark:linkbeta} and the theorem follows
from Theorem \ref{theorem:stran} and the Livsic theorem.
\end{proof}

\medskip

\begin{proof}[Proof of Theorem \ref{thm_main_imstar}]  This follows directly from Theorems \ref{theorem:control}
and \ref{theorem:sur2} and Remark \ref{remark:linkbeta}. The smoothness of the appropriate Fourier components of $w$
follows as in the proof of Theorem \ref{thm_main_i0star} using the ellipticity of $\eta_{\pm}$.
\end{proof}

\section{Examples} \label{sec_examples}
In this section we explain how we can perform alterations to the examples in \cite{G} to prove the following proposition:

\begin{Proposition} There are examples of closed orientable surfaces
with $\beta_{Ter}<2$, but arbitrarily close to $2$.
Moreover, for these examples
there are no geodesics trapped in the region of zero Gaussian curvature.
\end{Proposition}

\begin{proof}The construction in \cite{G} has some parameters that can be adjusted
to suit our purposes. Following the notation in \cite{G}, consider positive constants $b$ and $r_1$ such that $br_1<\pi/2$. There exists
a unique $r_2$, $0<r_2<r_1$, so that $ b^{-1}\sin br_1=\sinh(r_1-r_2)$.
Now choose $\varepsilon>0$ small enough so that
$\varepsilon<r_1-r_2$ and $b(r_1+\varepsilon)<\pi/2$.
Define $r_3:=r_{1}+\varepsilon$.

The main construction in \cite{G} ensures that given any $R>r_3-r_2$ we can construct an orientable closed surface $(M,g)$ with the following properties:
\begin{enumerate}
\item There is a point $p$ such that if $D$ denotes the ball centered at $p$
with radius $r_3$, then any geodesic segment in $D$ has length at most $2r_3$.
Moreover, the Gaussian curvature of $D$ is $\leq b^2$ and on the ball
of radius $r_1-\varepsilon$ centered at $p$ the curvature is constant and equal to $b^2$.
\item Outside $D$ the curvature equals $-1$.
\item Let $Q$ denote the annulus centered at $p$ with inner radius $r_3$ and outer radius $R+r_2$. Then the distance from $p$ to $\gamma(s)$ (where $\gamma$
is a unit speed geodesic) is a convex function of $s$ as long as $\gamma$ remains
in $D\cup Q$. Thus after leaving $D$, $\gamma$ must cross $Q$ to its outer boundary travelling at least a distance $R':=R+r_2-r_3$.
\end{enumerate}

In other words, the Gaussian curvature along $\gamma$ is at most $b^2$ for $s$ in certain intervals of length at most $2r_3$; these intervals are separated by intervals in which the curvature is $-1$ each of length at least $R'$.

Gulliver shows in \cite[p. 196]{G} that if
\[ b\tan br_3<\tanh R'\]
then $(M,g)$ has no conjugate points. Exactly the same proof shows that
if ($\beta>1$)
\begin{align}
&\sqrt{\beta} br_3<\pi/2, \label{eq:cond1}\\
&b\tan \sqrt{\beta} br_3<\tanh\sqrt{\beta} R',\label{eq:cond2}
\end{align}
then $\Psi_{t}^{\beta}$ is free of conjugate points.

Since the curvature is constant and equal to $b^2$ on the ball of radius
$r_1-\varepsilon$, it follows easily that if $b(r_1-\varepsilon)>\pi/2\sqrt{2}$, then the $2$-Jacobi equation has conjugate points and $\beta_{Ter}<2$.
Note that this also implies that $(M,g)$ has focal points.

Now given any $\beta\in (3/2,2)$ select $b>0$ and $\delta>0$ small enough
such that
\begin{align}
&\sqrt{\beta}(\pi/2\sqrt{2}+2b\delta)<\pi/2,\label{eq:cond3}\\
&b\tan(\sqrt{\beta}(\pi/2\sqrt{2}+2b\delta))<1/2 \label{eq:cond4}.
\end{align}
Define 
\[r_{1}:=\frac{\pi}{2\sqrt{2}b}+\delta.\]
With these choices of $b$ and $r_1$, $r_2$ is defined as above
and we choose $\varepsilon<r_1-r_2$ small enough so that
$\varepsilon<\delta$.
Using (\ref{eq:cond3}) we see that
\[\frac{\sqrt{\beta}\pi}{2\sqrt{2}}<\sqrt{\beta}b(r_1-\varepsilon)<\sqrt{\beta}b(r_{1}+\varepsilon)<\frac{\pi}{2}.\]
This ensures that (\ref{eq:cond1}) holds and that $\beta_{Ter}<2$.
Finally select $R$ large enough so that 
\[\tanh\sqrt{\beta}R'>1/2.\]
This together with (\ref{eq:cond4}) ensures that (\ref{eq:cond2}) holds and
hence $\beta_{Ter}\geq \beta$.

\end{proof}

\begin{Remark}{\rm An inspection of the proof also shows the following: the set of values in $(1,\infty)$ which are realized as terminator values of closed
orientable surfaces is dense.}
\end{Remark}

\section{Proof of Theorem \ref{thm:I2}}\label{proofmain}

The first step in the proof consists in showing that for any two holomorphic (in the angular variable) distributions $u$, $v$ such that $u \in L^2_x H^{-s}_{\theta}$, $v \in L^2_x H^{-t}_{\theta}$ for some $s, t \geq 0$, it is possible to define their product as an element $w$ in $H^{-N-2}(SM)$ if $N$ is sufficiently large.

\begin{Theorem} Let $(M,g)$ be a closed oriented surface. Suppose $u, v$ are distributions in $SM$ of the form $u = \sum_{k=0}^{\infty} u_k$, $v = \sum_{k=0}^{\infty} v_k$, where $u \in L^2_x H^{-s}_{\theta}$, $v \in L^2_x H^{-t}_{\theta}$ for some $s, t \geq 0$. Define 
$$
w_k = \sum_{j=0}^k u_j v_{k-j}, \quad k=0,1,\ldots.
$$
If $N$ is an integer with $N > s+t+1/2$, the sum $\sum_{k=0}^{\infty} w_k$ converges in $H^{-N-2}(SM)$ to some $w$ with $\norm{w}_{H^{-N-2}} \leq C \norm{u}_{L^2_x H^{-s}_{\theta}} \norm{v}_{L^2_x H^{-t}_{\theta}}$. Furthermore, 
\begin{equation} \label{wk_lone_estimate}
\norm{w_k}_{L^1(SM)} \leq \langle k \rangle^{s+t} \norm{u}_{L^2_x H^{-s}_{\theta}} \norm{v}_{L^2_x H^{-t}_{\theta}}.
\end{equation}
If $Xu = Xv = 0$, then also $Xw = 0$.
\label{theorem:product}
\end{Theorem}
\begin{proof}
One has $u_k, v_k \in L^2(SM)$, so each $w_k$ is in $L^1(SM)$. Note that 
$$
\sum_{j=0}^k \norm{u_j}_{L^2}^2 = \sum_{j=0}^k \langle j \rangle^{2s} \langle j \rangle^{-2s} \norm{u_j}_{L^2}^2 \leq \langle k \rangle^{2s} \norm{u}_{L^2_x H^{-s}_{\theta}}^2.
$$
Similarly $\sum_{j=0}^k \norm{v_j}_{L^2}^2 \leq \langle k \rangle^{2t} \norm{v}_{L^2_x H^{-t}_{\theta}}^2$. Consider the inner product space $(L^2(SM))^k$ with inner product 
$$
((a_0,\ldots,a_k), (b_0,\ldots,b_k)) = (a_0,b_0)_{L^2(SM)} + \ldots + (a_k,b_k)_{L^2(SM)}.
$$
The Cauchy-Schwarz inequality reads 
$$
\left| \int_{SM} (a_0 b_0 + \ldots + a_k b_k) \right| \leq \left( \sum_{j=0}^k \norm{a_j}_{L^2}^2 \right)^{1/2} \left( \sum_{j=0}^k \norm{b_j}_{L^2}^2 \right)^{1/2}.
$$
It follows that 
$$
\int_{SM} \abs{w_k} \leq \int_{SM} \sum_{j=0}^k \abs{u_j} \abs{v_{k-j}} \leq \left( \sum_{j=0}^k \norm{u_j}_{L^2}^2 \right)^{1/2} \left( \sum_{j=0}^k \norm{v_j}_{L^2}^2 \right)^{1/2}.
$$
This implies \eqref{wk_lone_estimate}.

Let $w^{(l)} = \sum_{j=0}^l w_j$, let $N$ be an integer with $N > s+t+1/2$, and let $\varphi$ be a function in $H^{N+2}(SM)$. Using \eqref{wk_lone_estimate}, we have 
\begin{align*}
\abs{\langle w^{(l)}, \varphi \rangle} &= \left| \sum_{j=0}^l \langle w_j, \varphi_j \rangle \right| \leq \sum_{j=0}^l \norm{w_j}_{L^1(SM)} \norm{\varphi_j}_{L^{\infty}(SM)} \\
 &\leq \norm{u}_{L^2_x H^{-s}_{\theta}} \norm{v}_{L^2_x H^{-t}_{\theta}} \sum_{j=0}^l \langle j \rangle^{s+t} \norm{\varphi_j}_{L^{\infty}(SM)}.
\end{align*}
By the Sobolev embedding $H^2(SM) \subset L^{\infty}(SM)$ and by Cauchy-Schwarz, we have 
$$
\sum_{j=0}^l \langle j \rangle^{s+t} \norm{\varphi_j}_{L^{\infty}(SM)} \leq C_{\delta} \left( \sum_{j=0}^l j^{2(s+t+\delta)} \norm{\varphi_j}_{H^2(SM)}^2 \right)^{1/2}
$$
for any $\delta > 1/2$. Choose $\delta=N-s-t$. Using an equivalent norm on $H^2(SM)$ involving $Y_1 = \eta_+$, $Y_2 = \eta_-$, and $Y_3 = V$, it follows that 
\begin{multline*}
\sum_{j=0}^l j^{2(s+t+\delta)} \norm{\varphi_j}_{H^2}^2 \leq \sum_{j=0}^l \left[ \norm{V^N \varphi_j}_{L^2}^2 + \sum_{q=1}^3 \norm{V^N Y_q \varphi_j}_{L^2}^2 + \sum_{q,r=1}^3 \norm{V^N Y_q Y_r \varphi_j}_{L^2}^2 \right] \\
 \leq \sum_{j=-2}^{l+2} \left[ \norm{(V^N \varphi)_j}_{L^2}^2 + \sum_{q=1}^3 \norm{(V^N Y_q \varphi)_j}_{L^2}^2 + \sum_{q,r=1}^3 \norm{(V^N Y_q Y_r \varphi)_j}_{L^2}^2 \right] \leq C \norm{\varphi}_{H^{N+2}}^2.
\end{multline*}
Thus $\norm{w^{(l)}}_{H^{-N-2}} \leq C \norm{u}_{L^2_x H^{-s}_{\theta}} \norm{v}_{L^2_x H^{-t}_{\theta}}$.

An argument using Cauchy sequences together with the previous computations shows that we may define 
$$
\langle w, \varphi \rangle = \lim_{l \to \infty} \langle w^{(l)}, \varphi \rangle, \quad \varphi \in H^{N+2}(SM).
$$
Then $w$ is an element of $H^{-N-2}(SM)$ with $\norm{w}_{H^{-N-2}} \leq C \norm{u}_{L^2_x H^{-s}_{\theta}} \norm{v}_{L^2_x H^{-t}_{\theta}}$.

The conditions $Xu = Xv = 0$ mean that $\eta_+ u_{k-1} + \eta_- u_{k+1} = 0$ for all $k$, and similarly for the $v_j$. Recall also that $u_k = v_k = 0$ for $k \leq -1$. We have $(Xw)_k = \eta_+ w_{k-1} + \eta_- w_{k+1}$, so $(Xw)_k = 0$ for $k \leq -2$. Also 
\begin{gather*}
(Xw)_{-1} = \eta_- w_0 = (\eta_- u_0) v_0 + u_0 (\eta_- v_0) = 0, \\
(Xw)_0 = (\eta_- u_0) v_1 + u_0 (\eta_- v_1) + (\eta_- u_1) v_0 + u_1 (\eta_- v_0) = 0.
\end{gather*}
Now if $l \geq 0$, 
$$
(Xw)_{l+1} = \eta_+ w_l + \eta_- w_{l+2} = \sum_{j=0}^l \eta_+(u_j v_{l-j}) + \sum_{j=0}^{l+2} \eta_-(u_j v_{l+2-j}) = 0.
$$
Thus $Xw = 0$.
\end{proof}

A combination of Theorems \ref{theorem:sur1'} and \ref{theorem:product} yields the following:

\begin{Theorem} Let $(M,g)$ be an Anosov surface. Suppose $q \in \Omega_2$ is in the linear span of $\{ a b \,;\, a, b \in \Omega_1 \text{ and } \eta_- a = \eta_- b = 0\}$. There exists $w = \sum_{k=2}^{\infty} w_k \in H^{-5}(SM)$ such that $Xw=0$, $w_2 = q$, $\norm{w}_{H^{-5}} \leq C \norm{q}_{L^2}$, and each $w_k$ is in $C^{\infty}(SM)$.
\label{theorem:sur3}
\end{Theorem}
\begin{proof}
Denote by $E$ the linear span of  $\{ a b \,;\, a, b \in \Omega_1 \text{ and } \eta_- a = \eta_- b = 0\}$. Then $E$ a subspace of the finite dimensional space $\{ q \in \Omega_2 \,;\, \eta_- q = 0 \}$, and $E$ has a basis $\{ a^{(1)} b^{(1)}, \ldots, a^{(N)} b^{(N)} \}$ where $\eta_- a^{(j)} = \eta_- b^{(j)} = 0$. 
By Theorem \ref{theorem:sur1'} there exist holomorphic distributions $u^{(j)}, v^{(j)} \in L^2_x H^{-1}_{\theta}$ such that $Xu^{(j)} = Xv^{(j)} = 0$, $u^{(j)}_1 = a^{(j)}$, $v^{(j)}_1 = b^{(j)}$, and 
$$
\norm{u^{(j)}}_{L^2_x H^{-1}_{\theta}} \leq C \norm{a^{(j)}}_{L^2}, \quad \norm{v^{(j)}}_{L^2_x H^{-1}_{\theta}} \leq C \norm{b^{(j)}}_{L^2}.
$$
Theorem \ref{theorem:product} implies that there are $w^{(j)} \in H^{-5}(SM)$ with $Xw^{(j)} = 0$, $w^{(j)} = \sum_{k=2}^{\infty} w^{(j)}_k$, $w^{(j)}_2 = a^{(j)} b^{(j)}$, and 
$$
\norm{w^{(j)}}_{H^{-5}} \leq C \norm{a^{(j)}}_{L^2} \norm{b^{(j)}}_{L^2}.
$$
The Fourier coefficients of $w^{(j)}$ are in $C^{\infty}(SM)$ since this is true for the Fourier coefficients of $u^{(j)}$ and $v^{(j)}$ (or alternatively by using the ellipticity of $\eta_-$).

Let now $q \in E$, be so that $q = \sum_{j=1}^N \lambda_j a^{(j)} b^{(j)}$ for some uniquely determined coefficients $\lambda_j \in \re$. Define $w = \sum_{j=1}^N \lambda_j w^{(j)}$. Then $w$ has all the required properties: the norm estimate holds since 
$$
\norm{w}_{H^{-5}} \leq C \sum_{j=1}^N \abs{\lambda_j}, \quad C = \sup_{j \in \{1,\ldots,n\}} \norm{w^{(j)}}_{H^{-5}},
$$
where the norm $\sum_{j=1}^N \abs{\lambda_j}$ is equivalent to $\norm{q}_{L^2}$ on the finite dimensional space $E$.
\end{proof}

Recall that $(M,g)$ has an underlying complex structure determined by $g$. We also recall that a Riemann surface $M$ is said be hyperelliptic if there is a holomorphic map $f:M\to S^2$ of degree two.
We are now ready to prove:

\begin{Theorem} Assume $(M,g)$ is an Anosov non hyperelliptic surface.
Let $f\in C^{\infty}(SM)$ be of the form $f=f_{-2}+f_0+f_{2}$.
Assume that there is $u\in C^{\infty}(SM)$ such that $Xu=f$.
Then $u_k=0$ for all $k$ with $|k|\geq 2$ and hence $f$ is potential.
\label{thm:withnh}
\end{Theorem}

\begin{proof} Without loss of generality we may assume that both $f$ and
$u$ are real-valued, otherwise split the transport equation into real and
imaginary parts. Then $\bar{f}_{k}=f_{-k}$ and $\bar{u}_{k}=u_{-k}$ for all
$k$.

Now observe that we have the following orthogonal decomposition:
\[\Omega_2=\eta_{+}(\Omega_{1})\oplus \mbox{\rm Ker}\,\eta_{-}.\]
Then $f_2=\eta_{+}(v_{1})+q_2$ where $q_2\in \mbox{\rm Ker}\,\eta_{-}$ and
$v_1\in \Omega_{1}$. Thus
\[Xv_1=\eta_{+}(v_{1})+\eta_{-}(v_{1})=f_2-q_{2}+\eta_{-}(v_{1}).\]
But $F:=-\eta_{-}(v_{1})+f_0\in \Omega_0$, therefore
\begin{equation}
X(u-v_1)=f_{-2}+F+q_2.
\label{eq:q2}
\end{equation}

Since $M$ is non hyperelliptic, we may use Max Noether's theorem \cite[p.159]{FK} which asserts that for any $m\geq 2$ the $m$-fold products of the abelian differentials of the first kind span the space of holomorphic $m$-differentials. This result for $m=2$ together with Lemma \ref{lemma:conformal} imply that $q_2$
is in the linear span of the set of products $a_1b_1$ where $a_1, b_1 \in \Omega_1$ and $\eta_- a_1 = \eta_- b_1 = 0$. 
By Theorem \ref{theorem:sur3} there is an invariant distribution $w=\sum_{k=2}^{\infty}w_k$
with $w_2=q_2$. Since $u-v_1\in C^{\infty}(SM)$, applying $w$ to equality (\ref{eq:q2}) we obtain
\[0=\langle w, X(u-v_1)\rangle=\langle w_2,q_2\rangle=\norm{q_{2}}^{2}_{L^{2}}.\]
Thus $q_2=0$. Since $f_{-2}=\bar{f}_{2}=\eta_{-}(\bar{v}_{1})$ we see using (\ref{eq:q2}) that
\[X(u-(v_{1}+\bar{v}_{1}))=f_{-2}+F-X(\bar{v}_{1})=F-\eta_{+}(\bar{v}_{1})\in\Omega_0.\]
Since $I_0$ is injective we derive that
$u-(v_{1}+\bar{v}_{1})$ must be constant and thus $u_k=0$ for all $k$ with $|k|\geq 2$.

\end{proof}

We now remove the assumption of being non hyperelliptic and we complete the proof of
Theorem \ref{thm:I2}.

\medskip

\begin{proof} [Proof of Theorem \ref{thm:I2}] It is well known that a closed Riemann surface $M$ of genus
$\tt{g}\geq 2$ admits normal covers of arbitrary degree. 
In other words given a positive integer $n$, there is a normal cover $N\mapsto M$
of degree $n$ and $N$ has genus $n(\tt{g}-1)+1$.
If $M$ is hyperelliptic, $N$ will be hyperelliptic only when $n=2,4$ \cite{Mac}, so
by taking $n\geq 5$ we can ensure that $N$ will not be hyperelliptic.

The metric $g$ can be lifted to $N$ and the geodesic flow continues to be
Anosov. The transport equation also lifts to $X\tilde{u}=\tilde{f}$, where
$\tilde{u}$ and $\tilde{f}$ are the lifts of $u$ and $f$.
We can now apply Theorem \ref{thm:withnh} in $N$ to deduce that
$\tilde{u}$ has degree one. Hence $u$ has degree one and $f$ is potential.

\end{proof}

\begin{Remark}{\rm To obtain solenoidal injectivity on tensors of order $m \geq 3$, it would be natural  to consider products of $m$ invariant distributions in $H^{-1}(SM)$ obtained from the surjectivity of $I_1^*$. However, even though the Fourier coefficients of such distributions are in $C^{\infty}(SM)$, we are currently unable to obtain the required estimates to show that the product makes sense as a distribution.}
\end{Remark}

\end{document}